\documentclass{scrartcl}

\usepackage[utf8]{inputenc}
\usepackage[T1]{fontenc}

\usepackage{graphicx}
\usepackage{amsmath,amssymb,amsthm}
\usepackage{esint} 
\usepackage{enumitem,color}
\usepackage{comment}
\usepackage{hyperref}
\usepackage{diagbox}
\usepackage{tikz}

\usepackage[style=alphabetic]{biblatex}
\AtEveryBibitem{\clearfield{issn}}
\newtheorem{thm}{Theorem}[section]
\newtheorem{prop}[thm]{Proposition}
\newtheorem{lem}[thm]{Lemma}
\newtheorem{cor}[thm]{Corollary}
\theoremstyle{definition}
\newtheorem{Def}[thm]{Definition}

\theoremstyle{remark}
\newtheorem{rem}[thm]{Remark}

\newcommand{\fig}[2]{\includegraphics[width=#1\textwidth]{#2}}

\newcommand{\Z}{\mathbb{Z}}

\title{%
  Enumeration of planar bipartite tight irreducible maps}
\author{J\'er\'emie Bouttier\thanks{Sorbonne Université and Université Paris Cité, CNRS, IMJ-PRG, F-75005 Paris, France} \and
  Emmanuel Guitter\thanks{Université Paris-Saclay, CNRS, CEA, Institut de physique
    théorique, 91191, Gif-sur-Yvette, France} \and 
  Hugo Manet\thanks{Université Paris Cité, CNRS, IRIF, F-75013, Paris, France}}
\date{\today}

\begin{document}
\maketitle

\begin{abstract}
  We consider planar bipartite maps which are both \emph{tight}, i.e.\ without vertices of degree $1$,
  and \emph{$2b$-irreducible}, i.e.\ such that each cycle has length at least $2b$ and such that any cycle of length exactly $2b$ is the contour of a face.
  It was shown by Budd that the number $\mathcal N_n^{(b)}$ of such maps made out of a fixed set of $n$ faces with prescribed even degrees is a polynomial in both $b$ and the face degrees. In this paper, we give an explicit expression for $\mathcal N_n^{(b)}$ by a direct bijective approach based on the so-called slice decomposition. More precisely, we decompose any of the maps at hand into a collection of $2b$-irreducible tight slices and a suitable two-face map.
  We show how to bijectively encode each $2b$-irreducible slice via a \emph{$b$-decorated tree} drawn on its derived map,
  and how to enumerate collections thereof.
  We then discuss the polynomial counting of two-face maps, and show how to combine it with the former enumeration to obtain $\mathcal N_n^{(b)}$.
\end{abstract}

\tableofcontents

\section{Introduction}
\label{sec:intro}

The study of random maps has been a subject of constant interest over the last sixty years, ever since Tutte’s first papers on the subject.
Of particular interest is the number of maps of fixed genus $g$ and with $n$ labeled faces of prescribed degrees:
explicit formulas were given by Tutte as early as 1961 in \cite{Tutte62} in the case of planar (i.e.\ genus $0$) maps with at most two faces of odd degree.
This formula was extended very recently to the case of planar maps with an arbitrary (necessarily even) number of odd-degree faces in \cite{polytightmaps}.

It was also recently realized that enumeration formulas remain simple in the case of \emph{tight} maps, i.e.\ maps without vertices of degree $1$. 
It was shown by Norbury in \cite{Norbury10, Norbury13} that the number of tight maps of fixed genus $g$ with $n$ labeled faces of respective degree $b_1, \ldots, b_n$ 
is in fact a \emph{quasi-polynomial} of degree $2n+6g-6$ in the $b_i$’s depending on their parity (provided that $n\geq 3$ if $g=0$). 
An explicit expression for this quasi-polynomial was given in \cite{polytightmaps} in the case $g=0$.

A remarkable extension of Norbury’s result was obtained by Budd in
\cite{Budd2022} for the case of \emph{essentially $2b$-irreducible}
maps with even degrees $b_i = 2m_i$.  By essentially $2b$-irreducible,
we mean maps that have no contractible cycle of length less than $2b$
and any contractible cycle of length $2b$ is the contour of a face of
degree $2b$. It was shown that the number of such maps is now a
polynomial of degree $2n+6g-6$ in both $b$ and the $m_i$’s. One of
Budd's motivations was to consider the limit where $b$ and the $m_i$'s
are taken to be large, which corresponds to considering so-called
irreducible \emph{metric maps}, having an unexpected connection with
Weil-Petersson volumes of hyperbolic surfaces~\cite{Budd2022a}.

The systematic study of planar irreducible maps, or more generally maps with a prescribed girth (which is the shortest length of a cycle in the map),
was initiated by Bernardi and Fusy in \cite{BF12a, BF12b}
via the existence of a canonical bi-orientation of such maps. In a later work, it was shown in \cite{irredmaps, irredsuite} how to recover 
their results by a substitution approach or, alternatively, via the decomposition of irreducible maps into \emph{slices} upon cutting 
these maps along properly chosen geodesic paths. Let us also mention the paper~\cite{Albenque2015} which develops an approach based on blossoming trees.

In \cite{Budd2022}, Budd relies precisely on the substitution approach of \cite{irredmaps} which, in the case of even-degree faces,
he generalizes to maps having arbitrary genus and adapts to deal with tight maps. 
The purpose of the present paper is, in the case $g=0$, to recover and sharpen the results of \cite{Budd2022}
by using instead the slice decomposition approach.
This allows us to write a slightly more explicit expression for the polynomials counting $2b$-irreducible maps with prescribed even degrees,
and to give a combinatorial interpretation of the various terms in that expression.
Our main result consists in the following theorem:
\begin{thm}
  \label{thm:main-result} Let $n, b, m_1,\ldots,m_n$ be positive integers and let us denote by $\mathcal{N}_{n}^{(b)}(2m_1,\allowbreak \ldots,2m_n)$ the number
of planar bipartite tight $2b$-irreducible maps with $n$ labeled faces
of respective degrees $2m_1,\ldots,2m_n$. Then, for $n \geq 3$, $m_1 \geq b+1$ and $m_2,\ldots,m_n\geq b$,
 we have
  \begin{equation}
    \label{eq:main-result}
\mathcal{N}_{n}^{(b)}(2m_1,\ldots,2m_n)=    (n-3)! \sum_{k_1,\ldots,k_n \geq 0} p^{(b)}_{k_1}(m_1) q^{(b)}_{k_2}(m_2) \cdots q^{(b)}_{k_n}(m_n) \alpha^{(b)}_{k_1+\cdots+k_n,n-3}
  \end{equation}
  where $p^{(b)}_{k}(m)$ and $q^{(b)}_{k}(m)$ denote the polynomials in $b$ and $m$:
  \begin{equation}
    \label{eq:pkb}
    p^{(b)}_{k}(m) := \binom{m-b-1}{k} \binom{m+b+k}{k} = \frac1{(k!)^2} \prod_{i=1}^{k} \left( m^2 - (b+i)^2 \right)
  \end{equation}
  \begin{equation}
    \label{eq:qkb}
      q^{(b)}_{k}(m) := \binom{m+b}{k} \binom{m-b-1+k}{k} = \frac1{(k!)^2} \prod_{i=0}^{k-1} \left( m^2 - (b-i)^2 \right)
  \end{equation}
  and $\alpha^{(b)}_{k,n}$ the polynomial in $b$ given by the expression:
  \begin{equation}
    \alpha_{k,n}^{(b)} = [u^{n-k}] \frac{1}{\left( 1 - u \sum\limits_{j=2}^b \frac1 b \binom{b}j \binom{b}{j-1} (-u)^{j-2} \right)^{n+1}}.
    \label{eq:alpha-k-n}
  \end{equation}
  The quantity  $\mathcal{N}_{n}^{(b)}(2m_1,\ldots,2m_n)$ is a polynomial in $b$ and $m_1,\ldots,m_n$, of total degree $2n-6$. It is symmetric in the variables $m_1,\ldots,m_n$, and even in each of them.
\end{thm}
\noindent Some remarks are in order. First, note that
$\alpha_{k,n}^{(b)}=0$ for $k>n$, hence the sum
in~\eqref{eq:main-result} is a \emph{finite sum}.  Second, the fact
that $\alpha_{k,n}^{(b)}$ is a polynomial in $b$ can be seen directly
from \eqref{eq:alpha-k-n} by noting that the fraction in the
right-hand side is a series in $u$ whose coefficients are polynomial
in $b$. We refer to Section~\ref{sec:arrowenum} for a more detailed
discussion, and in particular to Proposition~\ref{prop:alphapolexp}
below for a manifestly polynomial expression of $\alpha_{k,n}^{(b)}$.
Third, it follows from \cite{polytightmaps} that
Theorem~\ref{thm:main-result} also holds for $b=0$ upon understanding
Equation~\eqref{eq:alpha-k-n} as $\alpha_{k,n}^{(0)}=\delta_{k,n}$ and
noting that every map is $0$-irreducible.  Finally,
Equation~\eqref{eq:main-result} does not quite hold if we take all
$m_i$ equal to $b$: for $n\geq 4$, one needs to add an extra
pathological term $\frac{(n-1)!}{2}(-1)^n$, see \cite[Theorem
1]{Budd2022} and Appendix~\ref{sec:Npatho}.

\paragraph{Outline.}
Let us now discuss the path to Theorem~\ref{thm:main-result}.
It is a consequence of a chain of bijective decompositions described in Section~\ref{sec:slicedec}.
The first idea, borrowed from \cite{polytightmaps}, consists in decomposing a (at this stage, not necessarily irreducible nor tight) planar bipartite map into a two-face map with face degrees $2m_1, 2m_2$ and a collection of slices,
which are so to say pieces of maps lying in-between two geodesics, built out of the faces with degrees $2m_3, \ldots, 2m_n$.
This decomposition is recalled in Section~\ref{sec:slicedecgen}, where we also show how to extend it to the case of tight irreducible maps by imposing simple conditions on the two-face map and independent irreducibility and tightness conditions on the slices.
Then, following \cite{irredmaps}, we explain in Section~\ref{sec:quasislices} how a tight \emph{irreducible slice} can be decomposed recursively.
We show in Section~\ref{sec:treeformulation} that this decomposition has a bijective representation involving a \emph{decorated tree} which spans the \emph{derived map} of the slice.
A decorated tree is then naturally decomposed into its components living on the \emph{primal map}, which we call \emph{arrow trees},
and \emph{blossoming vertices} which are isolated dual vertices with the same (prescribed) degrees as their corresponding primal faces, among $2m_3, \ldots, 2m_n$.
Arrow trees and blossoming vertices are studied in Section~\ref{sec:treecomponents}.
As seen in Section~\ref{sec:tightchar}, the tightness property can be pulled back from the derived map onto the decorated tree, specifically as a property of its blossoming vertices.

Section~\ref{sec:enum} is devoted to the enumerative consequences of the above chain of decompositions.
In practice, the initial problem of enumerating planar bipartite tight $2b$-irreducible maps with prescribed face degrees boils down to two separate counting problems.
On the one hand, the problem of counting collections of tight irreducible slices; on the other hand, that of counting two-face maps.
The first problem reduces to counting collections of decorated trees.
This is performed in two steps. We first count arrow trees in Section~\ref{sec:arrowenum} in two different ways,
eventually yielding Equation~\eqref{eq:alpha-k-n} for their contribution $\alpha_{k,n}^{(b)}$ to formula~\eqref{eq:main-result}, as well as a recursive way to compute it.
We then evaluate in Section~\ref{sec:blossenum} the number of ways to connect the arrow trees via blossoming vertices into the desired collection of decorated trees.
As it turns out, each individual configuration of a blossoming vertex is counted polynomially in $b$ and in the $m_i$ corresponding to its degree
through Equation~\eqref{eq:qkb}, and their connection with arrow trees amounts to a convolution which preserves this polynomiality.
The second problem, i.e.\ counting two-face maps, is addressed in Section~\ref{sec:unicycenum}.
The irreducibility constraint imposes that these two-face maps have a long enough cycle, which also yields a polynomial in $b, m_1, m_2$ for their enumeration.
Section~\ref{sec:lutfin} explains how to combine everything, namely how to attach the slices to the two-face maps.
This last step leads to the formula~\eqref{eq:main-result}, which is actually shown to be a totally symmetric polynomial in the $m_i$'s, which also depends polynomially on $b$.
We also explore there a number of particular instances of Theorem~\ref{thm:main-result}.

Section~\ref{sec:conc} gathers some concluding remarks, while extra
material may be found in the appendices.
Appendix~\ref{app:closing-tree} discusses how to reconstruct a slice
from its associated decorated tree, by a closing procedure.
Appendix~\ref{app:compatibility-Budd} checks the compatibility of
formula~\eqref{eq:main-result} with the expression obtained in
\cite{Budd2022}. Appendix~\ref{sec:Npatho} discusses the enumeration
of $2b$-irreducible $2b$-angulations, which falls just outside the
range of validity of Theorem~\ref{thm:main-result}.

\paragraph{Basic definitions.}
Let us start by introducing some terminology related to maps, we refer to \cite{Schaeffer15} for more details.
A \emph{planar map} (hereafter called a map for short) is a connected (multi)graph drawn on the sphere without edge crossings. Loops and multiple edges are allowed.
A map consists of vertices, edges and faces. It is customary to draw a planar map on the plane, this amounts to choosing one face as the \emph{outer face}. 
A \emph{corner} is the angular sector delimited by two consecutive edges incident to a same vertex, hence also incident to
a same face. The degree of a vertex or a face is its number of incident corners.
In this paper, we consider \emph{bipartite} maps: the vertices can be partitioned in two sets in such a way that every edge connects vertices from different sets. Equivalently,
for planar maps, this amounts to requiring that every face be of even degree.

A path on a map is a path on the sphere that consists of edges and vertices of the map. The \emph{length} of a path is its number of edges, counted with 
multiplicity. The path is said simple if it does not visit a vertex more than once (except at its endpoints for a simple closed path). A simple closed path of non-zero length is called a \emph{cycle}. 
The \emph{girth} of a map is the minimal length of a cycle on the map. For $d$ a non-negative integer, a map is said \emph{$d$-irreducible} if it has girth at least $d$,
and every cycle of length $d$ is the contour of a face (by contour of a face, we mean the closed path formed by its incident edges). Note that every map is $0$-irreducible.
As we consider bipartite maps, whose all cycles necessarily have even length, we will take $d$ an even integer and write $d=2b$.

Following \cite{polytightmaps}, we define a \emph{tight map} as a map with some of its vertices marked, which is such than any leaf (vertex of degree $1$) is marked. In particular,
a tight map having no marked vertex is a map without leaves.  

Given a map and two of its vertices $u,v$, the (graph) \emph{distance} between $u$ and $v$ is the minimal length of a path connecting them. Such a path
of minimal length is called a \emph{geodesic}.

\paragraph{Acknowledgements.} We thank Timothy Budd, Guillaume Chapuy,
Éric Fusy and Grégory Miermont for fruitful discussions related to
this work. We acknowledge financial support from the Agence Nationale
de la Recherche via the grants ANR-18-CE40-0033 ``Dimers'',
ANR-19-CE48-0011 ``Combiné'' and ANR-23-CE48-0018 ``CartesEtPlus''.

\section{Slice decomposition of (tight) irreducible maps}
\label{sec:slicedec}
\subsection{Slice decomposition of maps}
\label{sec:slicedecgen}
An \emph{elementary slice} is a planar map with one marked face, chosen as the \emph{outer face}, having one marked
incident vertex called the \emph{apex} and one marked incident edge called the \emph{base}, which satisfy the following constraints:
denoting by $A$ the apex and by $B$, $C$ the endpoints of the base (with the outer face appearing on the right when going from $B$ to $C$),
\begin{itemize}
\item the \emph{blue boundary}, defined as the portion $AB$ of the contour of the outer face when going from $A$ to $B$ with the outer face on the right is a geodesic,
\item  the \emph{red boundary}, defined as the portion $CA$ of the contour of the outer face when going from $C$ to $A$ with the outer face on the right is the unique geodesic between $C$ and $A$,
\item the apex is the only vertex common to the blue and red boundaries. 
\end{itemize}
\begin{figure}
  \centering
  \fig{.8}{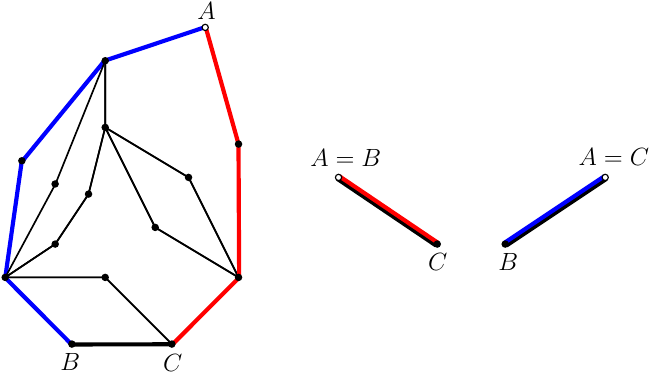}
  \caption{A generic slice (left) and the two possible slices without
    inner faces: the trivial slice (middle) and the empty slice
    (right).}
  \label{fig:slice-types}
\end{figure}
See Figure~\ref{fig:slice-types} for examples.
Note that, by the triangle inequality and the bipartiteness assumption, the length of $AB$ minus the length of $CA$ is equal to $\pm 1$. 
When this difference is equal to $-1$, then it follows from the above constraints that the whole slice is necessarily equal to the \emph{trivial slice},  
reduced to a single edge and two vertices, $A=B$ and $C$. 
In the following, we will call \emph{slice} for short an elementary slice which is not trivial.
Note that a slice may still be equal to the \emph{empty slice}, reduced to a single edge and two vertices, $A=C$ and $B$.
For $b$ a non-negative integer, we say that a slice is $2b$-irreducible if it has girth at least $2b$, and every cycle of length $2b$ is the contour of an \emph{inner} face
of the slice. 

\medskip The following proposition is a slight variant\footnote{In
  \cite[Proposition 4.7]{polytightmaps} it is assumed that the maps
  are tight. As explained in the proof of this proposition, the
  construction does not require tightness but is ``compatible'' with
  it. Here, we restate this compatibility property as the first item
  of Proposition~\ref{prop:tightdecomp}.}  of \cite[Proposition
4.7]{polytightmaps}:
\begin{prop}
\label{prop:annulbij}
Fix an integer $n\geq3$ and positive integers $m_1,\ldots,m_n$. 
There is a bijection between the set of planar bipartite maps with $n$ labeled faces of respective degrees $2m_1,\ldots,2m_n$, 
and the set of tuples of the form $(\mathbf{m}_{12},\mathbf{s}_1,\ldots,\mathbf{s}_{k+1})$, for some $k$ between $0$ and $n-3$, such that:
\begin{itemize}
\item $\mathbf{m}_{12}$ is a planar map with exactly two (labeled) faces of respective degrees $2m_1$ and $2m_2$, and with $(k+1)$ among its vertices marked,
one of them being distinguished,
\item $\mathbf{s}_i$ is a non-empty slice for every $i=1,\ldots,k+1$,  
\item there is a bijection between $\{3,\ldots,n\}$ and the union of the sets of the inner faces of $\mathbf{s}_1,\ldots,\mathbf{s}_{k+1}$, such
that 
each $j=3,\ldots,n$ is mapped to a face of degree $2m_j$, and $3$ is mapped to an inner face of $\mathbf{s}_1$.
\end{itemize} 
\end{prop}

\begin{rem}
  \label{rem:markvert}
  A slight extension of this bijection applies to maps $\mathbf{m}$ which, in
  addition to the labeling of their faces, have some of their vertices
  marked (these can be intuitively regarded as ``faces of zero
  degree''). Such maps still correspond to tuples
  $(\mathbf{m}_{12},\mathbf{s}_1,\ldots,\mathbf{s}_{k+1})$ as above,
  but now each slice $\mathbf{s}_i$ may have some of its vertices not
  belonging to its red boundary marked, and may be reduced to the
  \emph{marked empty slice} (i.e.\ the empty slice having its non-apex
  vertex marked). Under this bijection, the number of marked vertices
  in the original map is equal to the total number of marked vertices
  in $\mathbf{s}_1,\ldots,\mathbf{s}_{k+1}$.
  In this paper, where we concentrate on irreducibility, we will always
  consider maps $\mathbf m$ without such faces of degree zero.
  \end{rem}

\begin{figure}
  \centering
  \fig{}{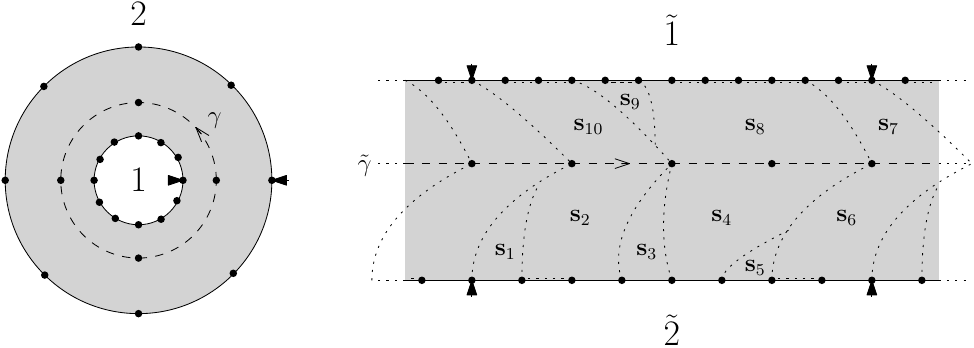}
  \caption{Sketch of the decomposition of a map $\mathbf{m}$ with two
    marked faces $1$ and $2$ (left) into elementary slices, by cutting
    its preimage $\tilde{\mathbf{m}}$ (right) along leftmost
    geodesics. In the notations of the main text, we have $m_1=6$,
    $m_2=4$ and $k=9$.}
  \label{fig:annuliftcutbis}
\end{figure}

We refer to \cite[Section 4.4]{polytightmaps} for a detailed
description of the bijection. To summarize, it consists in the
following steps, illustrated on Figure~\ref{fig:annuliftcutbis}.
\begin{enumerate}
\item We start from a map $\mathbf{m}$ with $n$ labeled faces of respective degrees $2m_1,\ldots,2m_n$ which we draw in the complex plane 
with the face $1$ containing the origin and the face $2$ chosen as the outer face\footnote{Note that we interchange the roles of faces $1$ and $2$ with respect to~\cite[Section 4.4]{polytightmaps}.}. We call \emph{separating girth} the minimal length of a \emph{separating cycle},
i.e.\ a cycle enclosing the origin. We denote by $\gamma$ the innermost separating cycle of length equal to the separating girth. By convention we orient $\gamma$ in the counterclockwise direction.
\item We consider the preimage $\tilde{\mathbf{m}}$ of $\mathbf{m}$ by the mapping $z\mapsto \exp(2 \mathrm{i} \pi z)$: it is an infinite map with two faces 
$\tilde{1}$ and $\tilde{2}$ of infinite degrees, which is invariant under the translation $z\mapsto z+1$. The minimal separating cycle $\gamma$ lifts to a biinfinite geodesic 
$\tilde{\gamma}=(\tilde{\gamma}_i)_{i\in \Z}$, oriented from $-\infty$ to $+\infty$.
\item We cut $\tilde{\mathbf{m}}$ along each leftmost geodesic going from a corner incident to one of the infinite faces $\tilde{1}$ and $\tilde{2}$ to $\tilde{\gamma}_i$ for $i$ large enough
(indeed such a leftmost geodesic eventually coalesces with $\tilde{\gamma}$).
This decomposes $\tilde{\mathbf{m}}$ into a collection of elementary slices, possibly trivial or empty. Upon restricting to an appropriate fundamental domain, this collection 
is finite, and each face $3\ldots,n$ of $\mathbf{m}$ corresponds to an inner face of degree $2m_i$ appearing in exactly one slice. If the map $\mathbf{m}$ carries marked vertices, each of them appears in exactly one slice deprived of its red boundary, which allows to transfer the markings canonically.
\item We let $(\mathbf{s}_1,\ldots,\mathbf{s}_{k+1})$ be the elementary slices in this decomposition which are neither trivial nor empty, where by convention $\mathbf{s}_1$
contains the face corresponding to face $3$, which allows us to list the other slices in some canonical way.
Note that, since each of the $\mathbf{s}_1,\ldots,\mathbf{s}_{k+1}$ contains at least a face, we have necessarily $k \leq n-3$.
\item By replacing the slices $(\mathbf{s}_1,\ldots,\mathbf{s}_{k+1})$ by marked empty slices (i.e., empty slices with the non-apex vertex marked), 
and performing the slice decomposition backwards, we obtain the two-face map $\mathbf{m}_{12}$ with its $k+1$ marked vertices. 
The marked empty slice replacing $\mathbf{s}_1$ gives rise to the distinguished marked vertex in $\mathbf{m}_{12}$.
\end{enumerate}

Let us record some useful properties of the above bijection in the
following proposition, which combines the discussions
of~\cite[Section~4.4]{polytightmaps} (regarding compatibility with
tightness) and~\cite[Section~9.3]{irredmaps} (regarding compatibility
with irreducibility and girth constraints).

\begin{prop} \label{prop:tightdecomp}
Let $\mathbf{m}$ be a planar map with $n$ labeled faces, let $(\mathbf{m}_{12},\mathbf{s}_1,\ldots,\mathbf{s}_{k+1})$ be its image
by the above bijection. Then:
\begin{itemize}
\item $\mathbf{m}$ is tight (i.e.\ has no leaves) if and only if all among $\mathbf{m}_{12},\mathbf{s}_1,\ldots,\mathbf{s}_{k+1}$ are tight
(note that $\mathbf{m}_{12}$ may have leaves provided they are marked), 
\item for any $b\geq 1$, $\mathbf{m}$ is \emph{essentially $2b$-irreducible} (i.e.\ every non-separating cycle has length at least $2b$ and
every such cycle of length $2b$ is the contour of a face) if and only if all among $\mathbf{s}_1,\ldots,\mathbf{s}_{k+1}$ are $2b$-irreducible,
\item the separating girth of the map $\mathbf{m}$ is equal to the length of the unique cycle of $\mathbf{m}_{12}$,
\item the contour of face $2$ is the unique minimal separating cycle in $\mathbf{m}$ if and only if the corresponding second face in $\mathbf{m}_{12}$ is simple and has no incident marked vertex.
\end{itemize}
From these properties, denoting by $2m_1$ and $2m_2$ the degrees of faces $1$ and $2$, we deduce that, for any $b\geq 1$:
\begin{itemize}
\item if $m_1,m_2>b$, $\mathbf{m}$ is $2b$-irreducible if and only if all among $\mathbf{s}_1,\ldots,\mathbf{s}_{k+1}$ are $2b$-irreducible,
and the length of the unique cycle of $\mathbf{m}_{12}$ is at least $2(b+1)$,
\item if $m_1>b$, $m_2=b$, $\mathbf{m}$ is $2b$-irreducible if and only if all among $\mathbf{s}_1,\ldots,\mathbf{s}_{k+1}$ are $2b$-irreducible,
  the length of the unique cycle of $\mathbf{m}_{12}$ is $2b$ (hence this cycle is the contour of face $2$),
  and none of the vertices of this cycle are marked.
\end{itemize}
\end{prop}

In view of Propositions~\ref{prop:annulbij} and \ref{prop:tightdecomp}, the problem of enumerating planar bipartite tight $2b$-irreducible maps is highly dependent on our ability to characterize tight $2b$-irreducible slices.
This is the purpose of the following sections where we show that $2b$-irreducible slices have a canonical decomposition which allows us to encode them by $b$-decorated plane trees, themselves formed of so-called \emph{$b$-arrow trees} attached to each other via \emph{blossoming vertices}. 

\begin{rem} \label{rem:singleslice} By specializing
  Propositions~\ref{prop:annulbij} and \ref{prop:tightdecomp} to the
  case $m_1=b+1$ and $m_2=b$, we find that planar tight
  $2b$-irreducible maps with $n$ labeled faces of respective degrees
  $2b+2,2b,m_3,\ldots,m_n$ are in bijection with tight
  $2b$-irreducible slices with $n-2$ labeled inner faces of respective
  degrees $m_3,\ldots,m_n$. Indeed, we note that the map
  $\mathbf{m}_{12}$ produced in the decomposition consists of a cycle
  of length $2b$ to which is attached a single edge leading to a
  single marked vertex, which forces $k=0$ hence a single slice is
  obtained.
\end{rem}

\subsection{Recursive decomposition of irreducible slices}
\label{sec:quasislices}
Let $b$ be a positive integer. As explained in details in \cite{irredmaps}, $2b$-irreducible slices can be decomposed recursively, at the price of introducing a slightly extended notion of slices. More precisely, for $p \geq 0$, we define a \emph{$2p$-slice} as   
a planar map with one marked face (the outer face) having one marked
incident vertex (the apex $A$) and one marked incident edge (the base $BC$) satisfying the following constraints:
\begin{itemize}
\item the blue boundary (defined as the portion $AB$ of the contour of the outer face when going from $A$ to $B$ with the outer face on the right) is a shortest path among all paths
connecting $B$ to $A$ \emph{which do not pass via the base},
\item  the red boundary (portion $CA$ of the contour of the outer face when going from $C$ to $A$ with the outer face on the right) is the unique geodesic between $C$ and $A$,
\item the apex is the only vertex common to the blue and red boundaries,
\item the length of $AB$ minus the length of $CA$ is equal to $2p+1$,
\item the slice has at least one inner face.
\end{itemize}
\begin{figure}
  \centering
  \fig{.7}{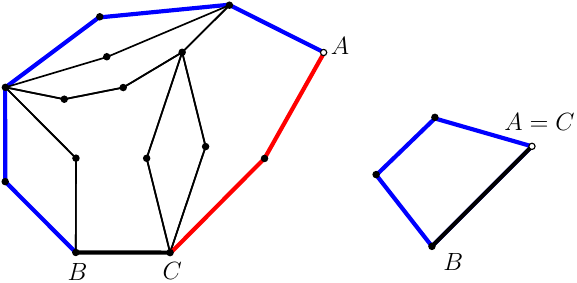}
  \caption{Examples of $2$-slices: a generic one (left)---note that it
    differs from the $0$-slice of Figure~\ref{fig:slice-types} only by
    a shift of the base edge---and the $4$-angle slice (right). Both
    are $4$-irreducible.}
  \label{fig:2p-slice-types}
\end{figure}
See Figure~\ref{fig:2p-slice-types} for examples.
Note that a $0$-slice is nothing but a non-trivial, non-empty elementary slice, and that the contour of the outer face of a $2p$-slice is simple.  A $2p$-slice
is said $2b$-irreducible if it has girth at least $2b$ and if every
cycle of length $2b$ is the contour of an inner face.

We now discuss
the precise recursive decomposition of a $2b$-irreducible $2p$-slice,
for $p \leq b$ (we will not need the case $p>b$ in this paper). This
requires us to distinguish three cases:
\begin{itemize}
  \item[(I)] when $0 \leq p \leq b-1$, except special case (III) below;
  \item[(II)] when $p=b$;
  \item[(III)] when $p = b-1$ and the outer face has degree $2b$.
\end{itemize}

\paragraph{Recursive decomposition of a $\boldsymbol{2b}$-irreducible $\boldsymbol{2p}$-slice, case  (I): $\boldsymbol{0\leq p\leq b-1}$.}
\begin{figure}[h]
  \centering
  \includegraphics{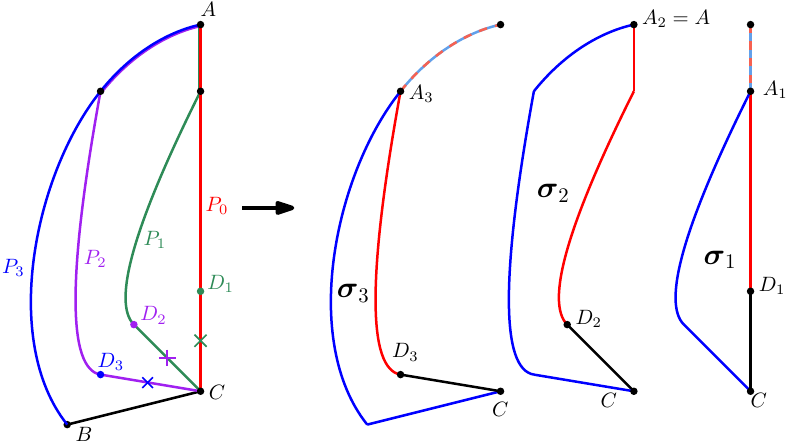}
  \caption{Case (I) of the decomposition.
    When $p \leq b-1$ and the slice is not reduced to the $2b$-angle slice, we know that $A \neq C$
    so we name $D_1$ the first vertex on the red boundary $P_0$.
    $P_1$ is the leftmost shortest path from C to A avoiding $CD_1$.
    Cutting along $P_1$ yields a $2p_1$-slice $\boldsymbol{\sigma}_1$ with apex $A_1$ and base edge $CD_1$.
    We start again in the rest of the slice (which is now a $2(p-p_1)$-slice with base $BC$),
    until $P_q$ goes through $BC$ and the blue boundary (here $q=3$).
  }
  \label{fig:step_II_primal}
\end{figure}

Take a $2b$-irreducible $2p$-slice $\boldsymbol{\sigma}$, as defined
just above, with $0\leq p\leq b-1$. The first step of the
decomposition, illustrated on Figure \ref{fig:step_II_primal}, is done
as follows: let us denote by $P_0$ the red boundary, travelled from
$C$ to $A$, and by $P_\infty$ the longer path from $C$ to $A$ obtained
by prefixing the blue boundary travelled from $B$ to $A$ with the base
edge travelled from $C$ to $B$. The lengths of $P_\infty$ and $P_0$
differ by $2p+2$, and their sum is equal to the degree of the outer
face. Using the $2b$-irreducibility constraint, we deduce that the
length of $P_0$ cannot be equal to $0$ unless we have $p=b-1$ and the
degree of the outer face is exactly $2b$: this case corresponds to a
unique configuration, called the \emph{$2b$-angle slice}, which will
be treated separately in case (III) below.  In all other cases, the
length of $P_0$ is at least $1$, and we consider the leftmost shortest
path $P_1$ among all paths from $C$ to $A$ \emph{which do not pass via
  the first edge $CD_1$ of $P_0$}.  Then, since $P_0$ is the unique
geodesic from $C$ to $A$, the difference between the length of $P_1$
and that of $P_0$ must be positive, and is even by bipartiteness: it
is equal to $2p_1$ for some $p_1\geq 1$. The part of
$\boldsymbol{\sigma}$ in-between $P_0$ and $P_1$ is then a
$2b$-irreducible $2p_1$-slice with base $CD_1$, which we denote by
$\boldsymbol{\sigma}_1$.  Note that we have $p_1 \leq p+1$ since $P_1$
is not longer than $P_\infty$, and if $p_1=p+1$ then $P_1=P_\infty$:
in this case, $\boldsymbol{\sigma}_1$ is the same map as
$\boldsymbol{\sigma}$, except that we have shifted the base edge by
one step to the right, so it becomes a $2(p+1)$-slice. For
$p_1 \leq p$, we continue the decomposition iteratively: denoting by
$D_2$ the endpoint of the first edge of $P_1$, we consider the
leftmost shortest path $P_2$ among all paths from $C$ to $A$ which
stay in-between $P_1$ and $P_\infty$ and do not pass via the edge
$CD_2$.  Then, the length of $P_2$ is equal to that of $P_1$ plus
$2p_2$ for some $p_2\geq 1$, and the part of the map in-between $P_1$
and $P_2$ is a $2b$-irreducible $2p_2$-slice with base $CD_2$, which
we denote by $\boldsymbol{\sigma}_2$. As $P_2$ is not longer than
$P_\infty$, we have $p_1+p_2 \leq p+1$, and in the case of equality we
have $P_2=P_\infty$, and we may stop the iteration. For
$p_1+p_2 \leq p$, we continue the iteration, defining a path $P_3$ and
a $2p_3$-slice $\boldsymbol{\sigma}_3$ with $p_3 \geq 1$ and
$p_1+p_2+p_3 \leq p+1$, and so on. Eventually, after $q$ iterations,
we will have $P_q=P_\infty$, and $p_1+\cdots+p_q = p+1$, and we stop
here. What we have done so far can be summarized into the following:
\begin{prop}
  \label{prop:decompI}
  For any integers $b, p$ with $0 \leq p \leq b-1$, there is a
  face-preserving\footnote{By \emph{face-preserving}, we mean that
    there is a degree-preserving bijection between the inner faces of
    $\boldsymbol{\sigma}$ and those of
    $\boldsymbol{\sigma}_1,\ldots,\boldsymbol{\sigma}_q$.}  bijection
  between the set of $2b$-irreducible $2p$-slices
  $\boldsymbol{\sigma}$ not equal to the $2b$-angle slice, and the set
  of sequences of the form
  $(\boldsymbol{\sigma}_1,\ldots,\boldsymbol{\sigma}_q)$ where $q$ is
  a positive integer and, for any $j=1,\ldots,q$,
  $\boldsymbol{\sigma}_j$ is a $2b$-irreducible $2p_j$-slice for some
  $p_j\geq 1$, with $p_1+\cdots+p_q=p+1$.
\end{prop}

The bijectivity can be checked by exhibiting the reverse bijection:
given a sequence
$(\boldsymbol{\sigma}_1,\ldots,\boldsymbol{\sigma}_q)$ as in the
proposition, it consists in gluing its elements into a single
$2p$-slice $\boldsymbol{\sigma}$. The key property is that
$2b$-irreducibility is preserved in this operation: in a nutshell this
is because we are gluing along geodesics, hence we cannot create
``short'' cycles. See \cite[Section 5.1]{irredmaps}.

We then continue the recursion by further decomposing the
$\boldsymbol{\sigma}_j$. Two situations may occur:
\begin{itemize}
\item if $q \geq 2$, or if $p<b-1$, then each $\boldsymbol{\sigma}_j$
  is a $2p_j$-slice with $p_j \leq b-1$: we may apply to it again the case
  (I) of the decomposition we have just described, or possibly the case
  (III) described below,
\item otherwise, for $q=1$ and $p=b-1$, we get a single $2b$-slice
  $\boldsymbol{\sigma}_1$: we apply to it the case (II) of the
  decomposition described below.
\end{itemize}
Let us observe that, in each ``branch'' of the recursion, we will
eventually arrive at either case (II) or case (III). Indeed, for
$q>2$, each $\boldsymbol{\sigma}_j$ contains fewer faces than
$\boldsymbol{\sigma}$ while, for $q=1$, we pass from a $2p$-slice to a
$2(p+1)$-slice. So, we end up with either the $2b$-angle slice or a
$2b$-irreducible $2b$-slice.

\begin{rem}
  \label{rem:decompIzero}
  This decomposition holds in particular for $p=0$, i.e.\ for
  non-empty $2b$-irreducible slices. In this case, we have
  necessarily\footnote{Unless we are in case (III), which can only
    happen if $0=b-1$ hence $b=1$.} $q=1$ and $p_1=1$. This
  corresponds to transforming the $0$-slice into a $2$-slice by
  changing its base from $BC$ to $CD_1$. For instance, applying the
  decomposition to the $0$-slice of Figure~\ref{fig:slice-types}-left,
  we obtain the $2$-slice of Figure~\ref{fig:2p-slice-types}-left,
  where $CD_1$ is renamed $BC$.
\end{rem}

\paragraph{Recursive decomposition of a $\boldsymbol{2b}$-irreducible $\boldsymbol{2p}$-slice, case (II): $\boldsymbol{p= b}$.}
\begin{figure}[h]
  \centering
  \includegraphics[width=.9143\textwidth]{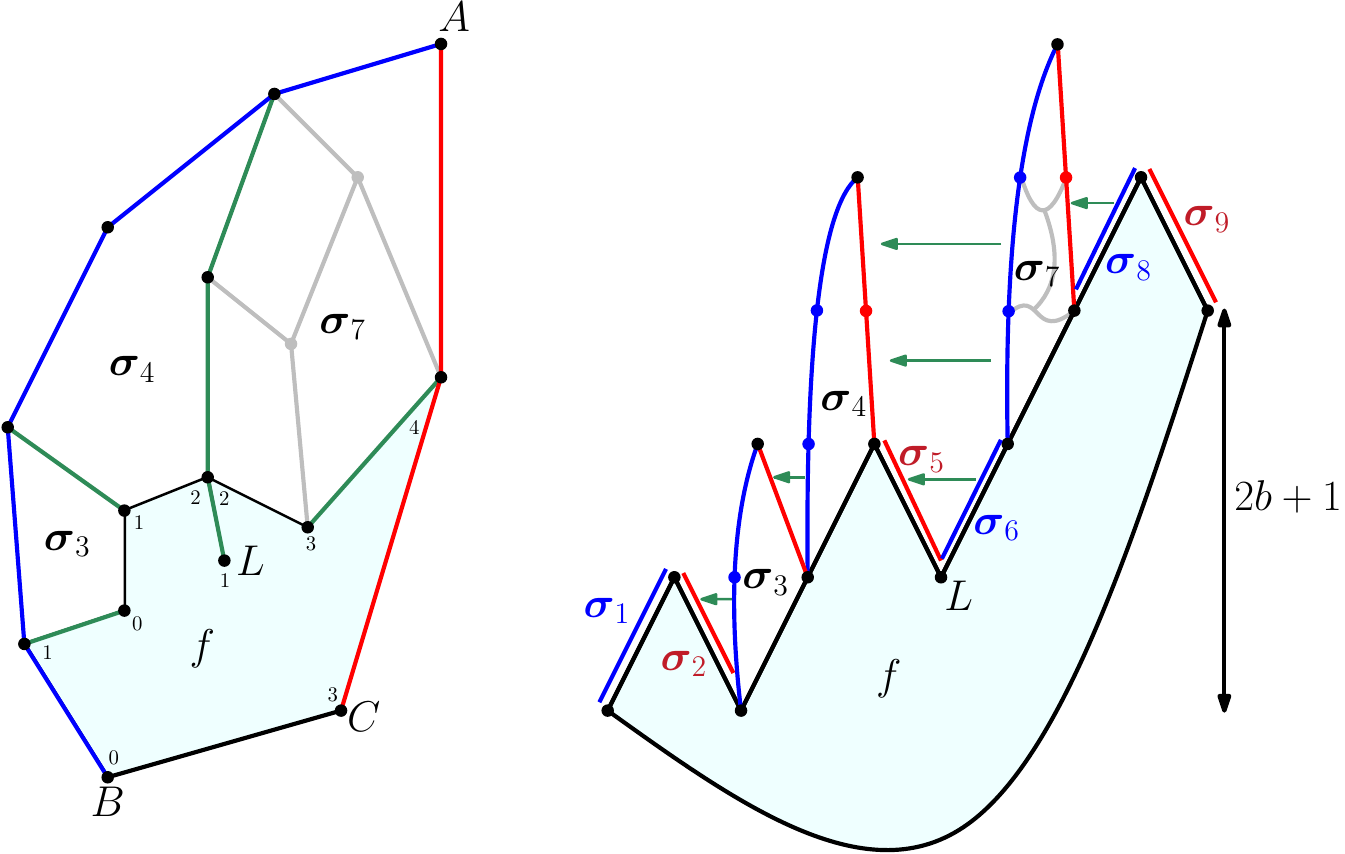}
  \caption{Case (II) of the decomposition, here for $p=b=1$.  This
    2-irreducible 2-slice is composed of a face $f$ of degree $10$,
    and $9$ elementary slices.  Left: we draw the (green) leftmost
    geodesic towards the apex from each vertex incident to $f$; this delimits the
    slices $\boldsymbol{\sigma}_1, \ldots, \boldsymbol{\sigma}_9$
    (here, the slices $\boldsymbol{\sigma}_1, \boldsymbol{\sigma}_6$
    and $\boldsymbol{\sigma}_8$ are empty, and
    $\boldsymbol{\sigma}_2, \boldsymbol{\sigma}_5$ and
    $\boldsymbol{\sigma}_9$ are trivial). The small labels in the
    corners of $f$ are the proximities to the apex, here equal to $5$ minus the length of a shortest path to $A$ avoiding $BC$: they go from $0$
    at $B$ to $2b+1=3$ at $C$ by steps of~$\pm 1$.  Right: the result
    after cutting along the geodesics. The green arrows show how to glue back the
    elementary slice boundaries, in order to recover the original
    slice.  Note that this slice is not tight, since it has a leaf $L$
    incident to $f$.}
  \label{fig:step_III_dual}
\end{figure}
Suppose now that we have a $2b$-irreducible $2b$-slice
$\boldsymbol{\sigma}$ with apex $A$ and base $BC$, and consider the
inner face $f$ immediately to the left of the base. Denoting by $2m$
the degree of $f$, consider its sequence of incident corners
$(c_0, \ldots c_{2m-1})$, as read clockwise around $f$ when going from
$B$ to $C$, and introduce the \emph{proximity} to the apex
$\ell_j:=d(A,B)- d(A,V_j)$, where $d(\cdot,\cdot)$ is the graph
distance in $\boldsymbol{\sigma}$ deprived of its base edge, and
$V_j$ is the vertex incident to $c_j$ for $j=0,\ldots, 2m-1$. See
Figure \ref{fig:step_III_dual} for an example. We have in particular
$\ell_0=0$ (since $V_0=B$), $\ell_{2m-1}=2b+1$ (by the definition of a
$2b$-slice), and $|\ell_j-\ell_{j-1}|=1$ for any
$j=1,\ldots,2m-1$. Note that this implies $m \geq b+1$. We may now cut
the slice along the leftmost geodesic from $V_j$ to $A$, for all $j$.
It is easily seen that the part of the map in-between the leftmost
geodesic from $V_{j-1}$ to $A$ and the leftmost geodesic from $V_j$ to
$A$ is a $2b$-irreducible elementary slice $\boldsymbol{\sigma}_j$
with base $V_{j-1}V_j$ for all $j\in \{1,\ldots,2m-1\}$. More
precisely, $\boldsymbol{\sigma}_j$ is the trivial slice whenever
$\ell_j-\ell_{j-1}=-1$ (this occurs $m-b-1$ times), while it is the
empty slice or a $2b$-irreducible $0$-slice whenever
$\ell_j-\ell_{j-1}=1$ (this occurs $m+b$ times). To summarize, we have
the following:
\begin{prop}
  \label{prop:decompIII}
  For any integers $m > b \geq 1$, there is a
  quasi-face-preserving\footnote{By this, we mean that all inner faces
    except the inner face of degree $2m$ incident to the base edge are
    preserved.}  bijection between the set of $2b$-irreducible
  $2b$-slices where the base edge is incident to an inner face of
  degree $2m$, and the set of $(2m-1)$-tuples of $2b$-irreducible
  elementary slices, exactly $m-b-1$ of which being equal to the
  trivial elementary slice. (Note that the $m+b$ remaining elementary
  slices are necessarily either equal to the empty slice, or to a
  $2b$-irreducible $0$-slice.)
\end{prop}

Again, the bijectivity can be checked by exhibiting the reverse
bijection. The most subtle point, already discussed
in~\cite{irredmaps}, is to check that this reverse bijection preserves
$2b$-irreducibility: again we use the fact that we are gluing slices
along geodesics, but we must also observe that, when ``recreating''
the base edge $BC$, we connect two vertices at graph distance at least
$2b+1$, so we cannot create a non-facial cycle of length $2b$. This
property would not be ensured if we applied decomposition (II) to a
$2p$-slice with $p<b$, and in particular to a $0$-slice. In
retrospect, this justifies why we need to introduce $2p$-slices for $p>0$.

Having decomposed the $2b$-slice $\boldsymbol{\sigma}$ as above, two
situations may occur:
\begin{itemize}
\item we only obtain trivial and empty slices: the recursive
  decomposition terminates here,
\item we obtain at least one $0$-slice: we apply to it the case (I) of
  the decomposition, or possibly the special case (III) if $b=1$ and
  the $0$-slice is a $2$-angle.
\end{itemize}

\paragraph{Recursive decomposition of a $\boldsymbol{2b}$-irreducible $\boldsymbol{2p}$-slice, case (III): the $\boldsymbol{2b}$-angle slice.}
We finally treat the special case where $p=b-1$ and the outer face is
of degree $2b$.  Since the contour of the outer face is a cycle, by
irreducibility it is also the contour of an \emph{inner} face $f$ of
degree $2b$.  The whole slice then consists of a single cycle of
length $2b$.  Note that, with $p=b-1$, the red boundary is reduced to
the vertex $C=A$, and the blue boundary comprises all edges apart from
the base.  We call this slice the \emph{$2b$-angle slice}.  Such slice
will be an atom in our recursive decomposition, which terminates here.

\medskip Altogether, combining steps (I), (II) and (III) decomposes
any $2b$-irreducible $2p$-slice with $0\leq p \leq b$ into pieces
which are either the trivial slice, the empty slice, or the $2b$-angle
slice (one may check that the recursion always terminates by induction
on the number of inner faces).  Figure \ref{fig:decomposition} shows
an example of full decomposition of a $0$-slice in the case $b=2$.

\begin{figure}[p]
  \centering
  \includegraphics[width=\textwidth]{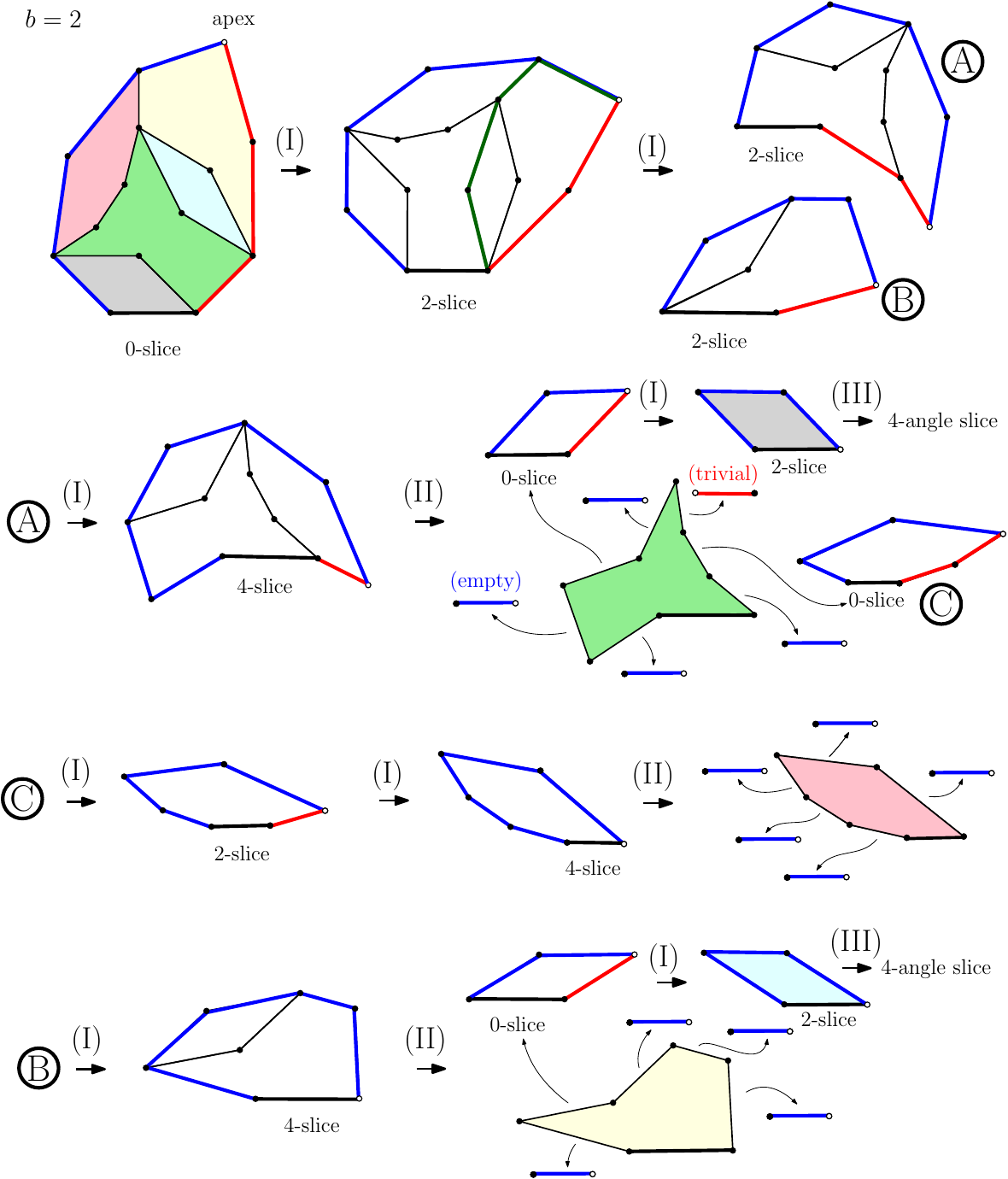}
  \caption{An example of decomposition of a $2b$-irreducible $0$-slice with $5$ faces, in the case $b=2$.
  In the second step (I) (upper line), the green geodesic cuts the $2$-slice in two sub-slices denoted by \textcircled{\raise-1pt\hbox{A}} and \textcircled{\raise-1pt\hbox{B}}. In the end, we are left with $1$ occurrence of the trivial slice, $13$ occurrences of the empty slice, and $2$ occurrences of the $2b$-angle slice.}
  \label{fig:decomposition}
\end{figure}

\subsection{Decorated tree formulation}
\label{sec:treeformulation}
Assume $b>0$ and $0 \leq p \leq b$ and consider a $2b$-irreducible $2p$-slice
$\boldsymbol{\sigma}$ to which we apply the above recursive
decomposition.  Following \cite{irredmaps,irredsuite}, it is useful to
encode this decomposition in the form of a tree, which we call
\emph{$b$-decorated tree} or \emph{decorated tree} for short, and we
will denote by $T(\boldsymbol{\sigma})$.  It turns out that
$T(\boldsymbol{\sigma})$ can be naturally drawn on the \emph{derived
  map} $\Delta(\boldsymbol{\sigma})$.

Recall from~\cite{Schaeffer15} that the derived map $\Delta(M)$ of a
map $M$ is the quadrangulation obtained by superimposing
$M$---hereafter called the primal map---with its dual map. The derived
map has three types of vertices, namely \emph{primal vertices},
\emph{dual vertices} and \emph{edge-vertices}, which are respectively
in bijection with the vertices, faces and edges of the primal map.
Precisely, if the primal map has two vertices $U$ and $V$ connected by
the edge $UV$ which has face $f$ on its left and face $g$ on its
right, then the derived map has an edge-vertex $\Delta(UV)$
corresponding to the edge $UV$, which is of degree $4$ and connected
(in clockwise order) to the vertices
$\Delta(U), \Delta(f), \Delta(V), \Delta(g)$.  Each edge of the
derived map connects an edge-vertex to either a primal or a dual vertex,
and hence corresponds to either a \emph{primal half-edge} or a
\emph{dual half-edge} accordingly.

In addition to being drawn on the derived map, the tree
$T(\boldsymbol{\sigma})$ carries some extra data, which we represent
in the form of \emph{decorations} as follows.
\begin{itemize}
\item Each primal half-edge belonging to $T(\boldsymbol{\sigma})$ and
  incident to a primal vertex of degree at least two in
  $T(\boldsymbol{\sigma})$ carries a number, ranging between $1$ and
  $b$, of \emph{arrows} pointing from the primal vertex to the
  edge-vertex.
\item Each dual vertex in $T(\boldsymbol{\sigma})$ may be incident, in
  addition to regular dual half-edges (which carry no arrow), to
  dangling half-edges which we call \emph{leaflets}.
\end{itemize}
The reader is invited to have a first look at
Figure~\ref{fig:decorated-tree-start}, which features these
decorations in the case $b=2$.

\bigskip Let us now explain how to construct $T(\boldsymbol{\sigma})$. In a
nutshell, we perform the recursive decomposition described in the
previous subsection, and build progressively the tree at each step,
according to specific rules described below.
We start with the following useful definition:
\begin{Def}
  Given a $2p$-slice $\boldsymbol{\sigma}$, with apex $A$, base $BC$, and outer face $f_0$,
  we define the \emph{tree vertex set} $S(\boldsymbol{\sigma})$
  as the set of all vertices of the derived map
  $\Delta(\boldsymbol{\sigma})$ deprived of $\Delta(f_0)$ and from the
  primal and edge-vertices corresponding to the vertices and edges of
  the blue boundary.
  In particular, neither $\Delta(A)$ nor $\Delta(B)$
  belong to $S(\boldsymbol{\sigma})$, but $\Delta(BC)$ and $\Delta(f)$
  do, where $f$ is the inner face incident to $BC$.
  We extend this definition by setting $S(\boldsymbol{\sigma})=\{\Delta(BC),\Delta(C)\}$ for $\boldsymbol{\sigma}$ the trivial slice,
  and $S(\boldsymbol{\sigma})=\varnothing$ for $\boldsymbol{\sigma}$ the empty slice.
\end{Def}
This allows us to state an ``invariant'' of the
recursion, which will be verified inductively:
\begin{prop}
  \label{prop:treeinvariant}
  Given a $2b$-irreducible $2p$-slice $\boldsymbol{\sigma}$ with
  $0 \leq p \leq b$, $T(\boldsymbol{\sigma})$ is a tree drawn on the
  derived map $\Delta(\boldsymbol{\sigma})$, having vertex set
  $S(\boldsymbol{\sigma})$. Every edge-vertex in
  $S(\boldsymbol{\sigma})$ has degree two in $T(\boldsymbol{\sigma})$,
  except $\Delta(BC)$ which has degree one, and which we choose as the root.
  For $p=b$, $\Delta(BC)$ is connected to $\Delta(f)$.
  For $p \leq b-1$, $\Delta(BC)$ is connected to $\Delta(C)$ by a primal half-edge
  carrying $b-p$ arrows, unless $\boldsymbol{\sigma}$ is equal to the $2b$-angle slice,
  in which case $\Delta(BC)$ is connected to $\Delta(f)$.
\end{prop}

\begin{figure}
  \centering
  \fig{.7}{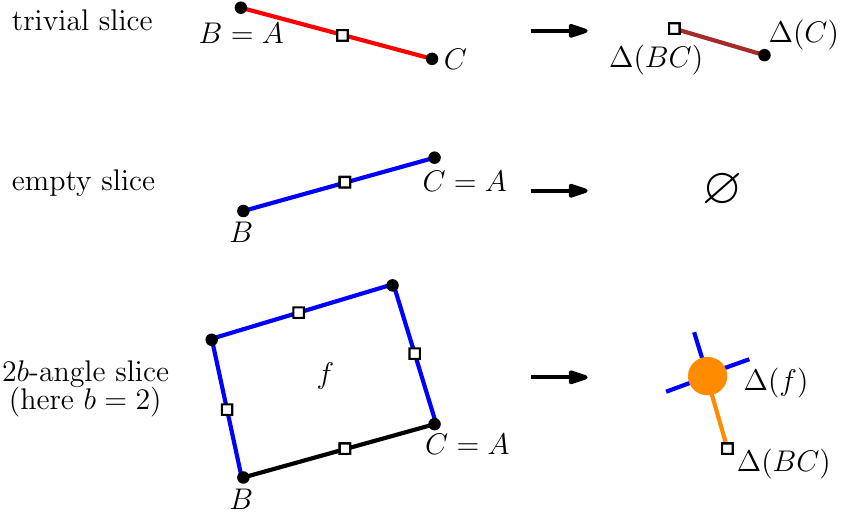}
  \caption{Initialisation of the recursive construction of $T(\boldsymbol{\sigma})$.
    Primal vertices and edge-vertices are represented as black dots
    and white squares, respectively. On the right, leaflets are shown
    in blue.}
  \label{fig:tree-construction-init}
\end{figure}

We now give the precise construction rules.
See Figure~\ref{fig:decorated-tree-start} for a full application of our construction.
\paragraph{Recursive construction of the decorated tree, initialisation.}
We define $T(\boldsymbol{\sigma})$ when
$\boldsymbol{\sigma}$ is an ``atom'' of our recursive decomposition, namely when it is equal
either to the trivial slice, to the empty slice, or to the $2b$-angle slice.
For the trivial slice, $T(\boldsymbol{\sigma})$ consists of
a single primal half-edge connecting $\Delta(BC)$ to $\Delta(C)$,
carrying no arrow. This tree is called the \emph{trivial tree}.
For the empty slice, $T(\boldsymbol{\sigma})$ is
defined as the empty graph containing no vertex.
For the $2b$-angle slice, corresponding to the case (III) discussed in the previous subsection,
$T(\boldsymbol{\sigma})$ consists of a single dual half-edge
connecting $\Delta(BC)$ to $\Delta(f)$ (with $f$ the inner face),
and $2b-1$ leaflets incident to $\Delta(f)$. See
Figure~\ref{fig:tree-construction-init} for an illustration. Note that
all these conventional definitions are consistent with the property
that $T(\boldsymbol{\sigma})$ has vertex set $S(\boldsymbol{\sigma})$.
We now turn to the recursive part of the
construction, for which we have to distinguish the cases (I) and (II)
discussed in the previous subsection.

\paragraph{Recursive construction of the decorated tree, case  (I): $\boldsymbol{0\leq p\leq b-1}$.}
\begin{figure}[t]
  \centering \includegraphics[width=\textwidth]{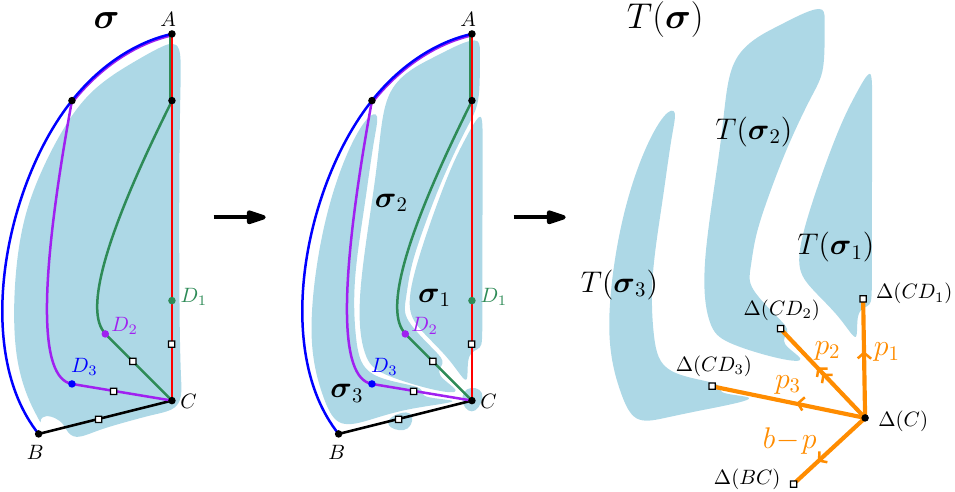}
  \caption{Tree formulation of the case (I) decomposition of
    Figure~\ref{fig:step_II_primal}. Primal vertices and edge-vertices
    are represented as black dots and white squares, respectively.
    Represented in light blue, the tree vertex set $S(\boldsymbol{\sigma})$ (left),
    and its decomposition as the disjoint union of the sets
    $\{\Delta(BC)\},\{\Delta(C)\},S(\boldsymbol{\sigma}_1),\ldots,S(\boldsymbol{\sigma}_q)$ (middle).
    On the right, the tree $T(\boldsymbol{\sigma})$ constructed in the text.
  }
  \label{fig:tree-construction-I}
\end{figure}
Suppose that we are in case (I) of the recursive decomposition. As
summarized in Proposition~\ref{prop:decompI}, $\boldsymbol{\sigma}$ is
then decomposed into a sequence
$(\boldsymbol{\sigma}_1,\ldots,\boldsymbol{\sigma}_q)$ where $q$ is a
positive integer and, for any $j=1,\ldots,q$, $\boldsymbol{\sigma}_j$
is a $2b$-irreducible $2p_j$-slice for some $p_j\geq 1$, with
$p_1+\cdots+p_q=p+1$. Then, $T(\boldsymbol{\sigma})$ consists of the
following elements (see Figure~\ref{fig:tree-construction-I} for an
illustration):
\begin{itemize}
\item the primal half-edge connecting $\Delta(BC)$ to $\Delta(C)$, on
  which we place $b-p$ arrows,
\item for each $j=1,\ldots,q$, the primal half-edge connecting
  $\Delta(C)$ to $\Delta(CD_j)$ (with $CD_j$ the base of
  $\boldsymbol{\sigma}_j$), on which we place $p_i$ arrows,
\item and the trees
  $T(\boldsymbol{\sigma}_1),\ldots,T(\boldsymbol{\sigma}_q)$ which we
  proceed to construct recursively.
\end{itemize}
Assuming that Proposition~\ref{prop:treeinvariant} holds for
$\boldsymbol{\sigma}_1,\ldots,\boldsymbol{\sigma}_q$, we may verify
that it also holds for $\boldsymbol{\sigma}$, by making the key
observation that the tree vertex set $S(\boldsymbol{\sigma})$ is the
disjoint union of the sets
$\{\Delta(BC)\},\{\Delta(C)\},S(\boldsymbol{\sigma}_1),\ldots,S(\boldsymbol{\sigma}_q)$. We
also observe that, in $T(\boldsymbol{\sigma})$, the primal vertex
$\Delta(C)$ has degree $q+1 \geq 2$, and that the total number of
arrows on its incident primal half-edges is equal to $(b-p) + p_1 + \cdots + p_q = b+1$.

\paragraph{Recursive construction of the decorated tree, case (II): $\boldsymbol{p= b}$.}
\begin{figure}[t]
  \centering
  \fig{.8}{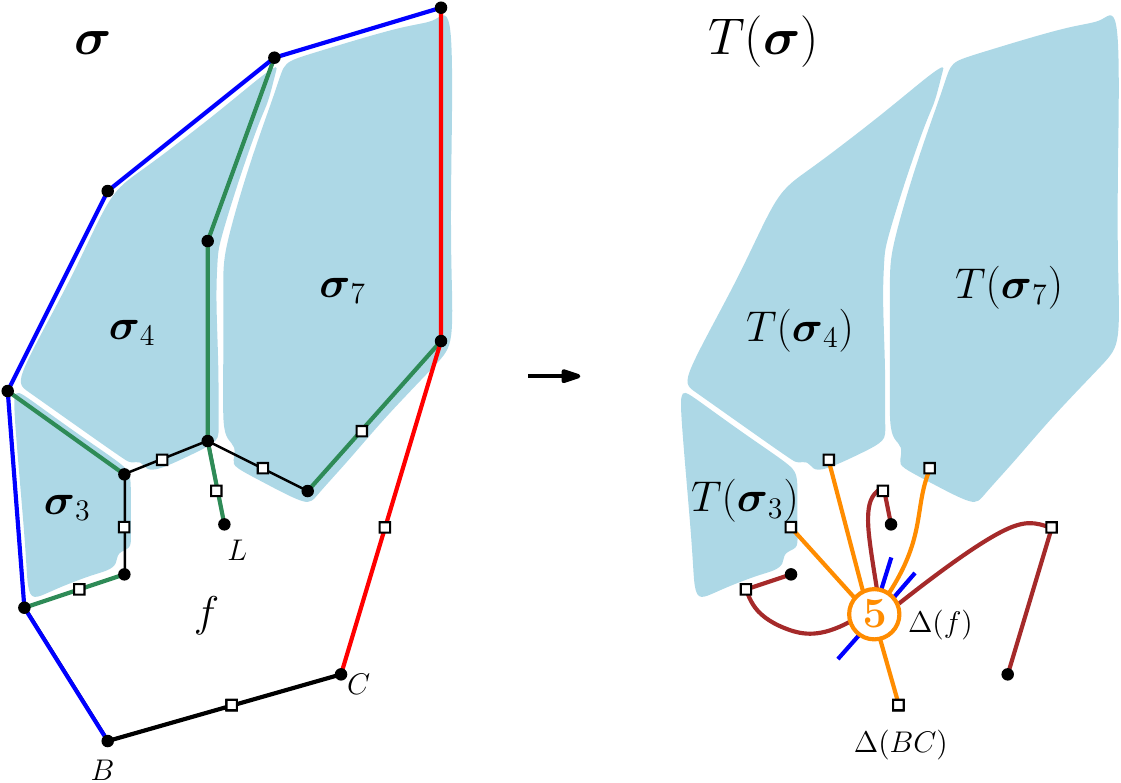}
  \caption{Tree formulation of the case (II) decomposition of
    Figure~\ref{fig:step_III_dual} (for $b=1$). In
    $T(\boldsymbol{\sigma})$, the empty slices
    $\boldsymbol{\sigma}_1,\boldsymbol{\sigma}_6,\boldsymbol{\sigma}_8$
    give rise to leaflets (shown in blue) while the trivial slices
    $\boldsymbol{\sigma}_2,\boldsymbol{\sigma}_5,\boldsymbol{\sigma}_9$
    give rise to twigs (shown in brown). For convenience we display on
    the dual vertex $\Delta(f)$ its half-degree, here equal to
    $m=5$. Note that the vertex $L$ of degree $1$ incident to $f$ in
    $\boldsymbol{\sigma}$ gives rise in $T(\boldsymbol{\sigma})$ to a
    twig followed by a leaflet, when going clockwise around
    $\Delta(f)$. Such pattern is forbidden in tight maps.}
  \label{fig:tree-construction-II}
\end{figure}
Suppose now that we are in case (II) of the recursive decomposition
and let $m$ be the half-degree of the inner face $f$. As summarized in
Proposition~\ref{prop:decompIII}, $\boldsymbol{\sigma}$ is then
decomposed into a tuple
$(\boldsymbol{\sigma}_1,\ldots,\boldsymbol{\sigma}_{2m-1})$ of
$2b$-irreducible elementary slices, exactly $m-b-1$ of which are
trivial. Then, $T(\boldsymbol{\sigma})$ consists of the following
elements (see Figure~\ref{fig:tree-construction-II} for an
illustration):
\begin{itemize}
\item the dual half-edge connecting $\Delta(BC)$ to $\Delta(f)$,
\item for each $j=1,\ldots,2m-1$, the dual half-edge connecting
  $\Delta(f)$ to the edge-vertex corresponding to the base of
  $\boldsymbol{\sigma}_j$, unless the latter is equal to the empty
  slice, in which case we replace the dual half-edge by a leaflet
  attached to $\Delta(f)$,
\item and the trees
  $T(\boldsymbol{\sigma}_1),\ldots,T(\boldsymbol{\sigma}_{2m-1})$ which we
  proceed to construct recursively.
\end{itemize}
Assuming that Proposition~\ref{prop:treeinvariant} holds for the
$\boldsymbol{\sigma}_j$ which are neither trivial nor empty, we may
verify that it also holds for $\boldsymbol{\sigma}$, by making the key
observation that the tree vertex set $S(\boldsymbol{\sigma})$ is the
disjoint union of the sets
$\{\Delta(BC),\Delta(f)\},S(\boldsymbol{\sigma}_1),\ldots,S(\boldsymbol{\sigma}_{2m-1})$. We also observe that, in
$T(\boldsymbol{\sigma})$, the dual vertex $\Delta(f)$ has degree $2m$,
accounting for the contribution of leaflets. It is incident to exactly
$m-b-1$ dual half-edges leading to an instance of the trivial tree (namely, the tree
corresponding to the trivial slice, see again
Figure~\ref{fig:tree-construction-init}). We call \emph{twig} the combination of such a
dual half-edge and its attached trivial tree,
so that $\Delta(f)$ is attached to exactly $m-b-1$ twigs. The
remaining contribution $m+b+1$ to the degree of $\Delta(f)$ comes from
leaflets, dual half-edges leading to non-trivial trees, and the
root dual half-edge coming from $\Delta(BC)$.

\begin{figure}[h]
  \centering
  \includegraphics[scale=1.2]{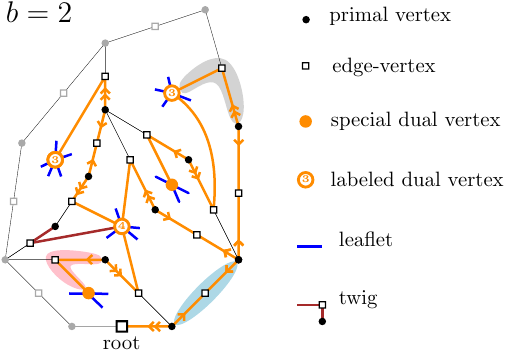}
  \caption{The decorated tree representation of the decomposition in Figure \ref{fig:decomposition}.
  Here we draw the tree superimposed on the slice.
  We indicated in light blue a bioriented edge-vertex,
  in gray a regular bent edge-vertex,
  and in pink a special bent edge-vertex.
  The primal and edge-vertices corresponding to the blue boundary of the slice
  are shown in grey, and do not belong to the tree vertex set.}
  \label{fig:decorated-tree-start}
\end{figure}

\paragraph{Characterization of $b$-decorated trees.}
Let us now give an intrinsic characterization of the trees that we obtain.
A $b$-decorated tree is a plane tree satisfying the following properties.
\begin{itemize}
  \item It is made of three types of vertices: primal, dual, and edge-vertices,
  connected by either \emph{primal half-edges} connecting a primal vertex to an edge-vertex,
  or \emph{dual half-edges} connecting a dual vertex to an edge-vertex.
  \item It carries two types of decorations:
    \begin{itemize}
      \item \emph{arrows}, in number between $1$ and $b$, placed on all primal half-edges incident to a primal vertex of degree at least two,
        and pointing away from that vertex;
      \item \emph{leaflets}, incident to dual vertices, that contribute to their degrees.
    \end{itemize}
  \item Around each primal vertex of degree at least two, the total number of arrows is equal to $b+1$.
  \item The tree is planted on an edge-vertex of degree one, hereafter called the root.
    All the edge-vertices different from the root have degree two.
  \item An edge-vertex incident to two primal half-edges is called a \emph{bioriented edge}:
    it must have both its incident half-edges carrying arrows, with $b$ arrows in total,
    and may therefore exist only when $b\geq 2$.
  \item An edge-vertex incident to two dual half-edges is called a \emph{dual/dual edge-vertex}:
    it may exist only when $b=1$, and must have exactly one of its adjacent dual vertices of degree $2b = 2$ (see Figure~\ref{fig:decomp-b-eq-1}).
  \item An edge-vertex incident to one primal half-edge and one dual half-edge is called a \emph{bent edge-vertex},
    which must be of one of the following types:
    \begin{itemize}
      \item a \emph{twig-vertex}: the primal half-edge carries no arrow, hence leads to a primal vertex of degree one.
        The ensemble made of the twig-vertex, its incident half-edges, and the adjacent primal vertex, form a \emph{twig};
      \item a \emph{special bent edge-vertex}: the primal half-edge carries $b-1$ arrows, and the adjacent dual vertex has degree $2b$;
      \item a \emph{regular bent edge-vertex}: the primal half-edge carries $b$ arrows, and the adjacent dual vertex has degree at least $2b+2$.
    \end{itemize}
  \item Each dual vertex of degree $2b$, hereafter called \emph{special dual vertex}, is incident to exactly one dual half-edge and $2b-1$ leaflets.
  \item Every other dual vertex has an even degree larger than $2b$, and is hereafter called \emph{labeled dual vertex}.
    A labeled dual vertex of degree $2m$ has label $m \geq b+1$, and is adjacent to exactly $m-b-1$ twig-vertices.
\end{itemize}
See again Figure~\ref{fig:decorated-tree-start} for an example of a $b$-decorated tree in the case $b=2$.
With this characterization at hand, we may state the following:
\begin{prop}
  For $b \geq 1$ and $p = 0, \ldots, b$, the mapping $\boldsymbol{\sigma} \mapsto T(\boldsymbol{\sigma})$ is a bijection
  between the set of $2b$-irreducible $2p$-slices different from the $2b$-angle slice, and the set of $b$-decorated trees such that
  the root edge-vertex is incident to a primal half-edge carrying $b-p$ arrows when $p \leq b-1$,
  or to a dual half-edge leading to a labeled dual vertex when $p=b$. 
  For each $m$, the number of inner faces of degree $2m$ in $\boldsymbol{\sigma}$ is equal to the number of dual vertices of degree $2m$ in $T(\boldsymbol{\sigma})$.
\end{prop}
This proposition may be proved by checking that the $b$-decorated trees have a recursive decomposition which is equivalent to that of $2b$-irreducible slices.
For completeness, we also give a self-contained description of the inverse bijection in Appendix~\ref{app:closing-tree}.

\begin{rem}
  \label{rem:altrepr}
  For completeness, let us mention an alternate but equivalent way to
  represent $b$-decorated trees, which is strongly reminiscent of the
  $\Z$-mobiles considered in~\cite{BF12b}. We still have three types
  of vertices (primal, dual and edge-vertices), connected by primal
  and dual half-edges. We still plant the tree on an edge-vertex of
  degree one, and every other edge-vertex has degree two. We still
  have leaflets\footnote{Leaflets correspond to buds in the
    terminology of~\cite{BF12b}.} attached to the dual vertices and
  contributing to their degree (which must be an even integer larger
  than or equal to $2b$). But instead of placing arrows on some primal
  half-edges, we now assign an integer \emph{weight} to \emph{every}
  (primal or dual) half-edge, with the following rules:
  \begin{itemize}
  \item the weight of a primal half-edge is an integer between $1$ and $b+1$,
  \item the weight of a dual half-edge is either $-1$, $0$ or $+1$,
  \item defining the \emph{total weight} of a vertex as the sum of
    the weights of its incident half-edges (ignoring leaflets):
    \begin{itemize}
    \item each primal vertex has total weight $b+1$,
    \item each non-root edge-vertex has total weight $b$,
    \item for every $m \geq b$, each dual vertex of degree $2m$ has
      total weight $b+1-m$; if $m=b$ it is incident to exactly one
      dual half-edge of weight $+1$ and $2b-1$ leaflets; if
      $m \geq b+1$ it has no incident dual half-edge of weight $+1$
      (hence has exactly $m-b-1$ incident dual half-edges of weight
      $-1$).
    \end{itemize}
  \end{itemize}
  We recover the previous representation by placing $i$ arrows on each
  primal half-edge of weight $i$, for every $i$ between $1$ and
  $b$. Note that the primal half-edges of weight $b+1$ and the dual
  half-edges of weight $-1$ only appear within twigs. Let us observe
  finally that, in the absence of dual vertices of degree $2b$ (which
  implies that there are no dual half-edges of weight $+1$), we
  recover precisely the $(b+1)$-dibranching mobiles as defined
  in~\cite[Definition 8]{BF12b}.
\end{rem}

\subsection{Arrow trees and blossoming vertices.}
\label{sec:treecomponents}
From the discussion of the previous subsections, the problem of enumerating $2b$-irreducible maps can be achieved by first enumerating $2b$-irreducible slices,
which in turn amounts to enumerating $b$-decorated trees.
In order to do so, it is useful to further decompose these decorated trees into more elementary components as follows (see Figure \ref{fig:decorated-tree-decomp}).
\begin{figure}[h]
  \centering
  \includegraphics[width=\textwidth]{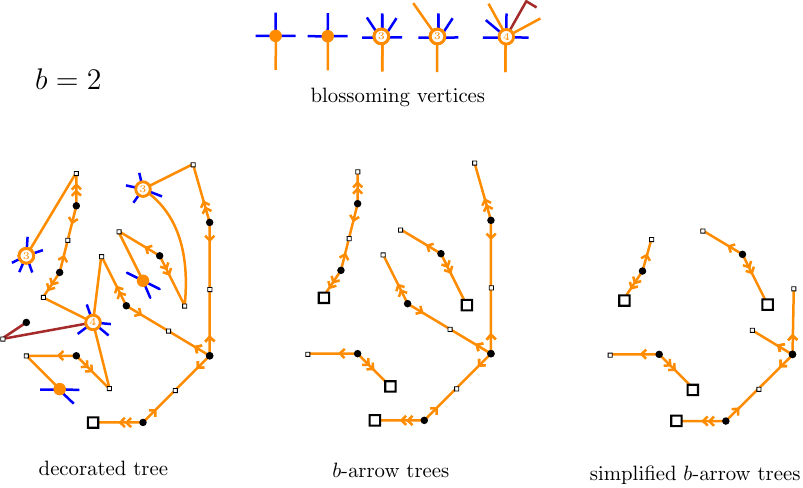}
  \caption{Decomposition of the decorated tree (left) of Figure~\ref{fig:decorated-tree-start}
    into its connected components made of $b$-arrow trees and blossoming vertices (middle).
    The $b$-arrow trees may be further reduced into simplified $b$-arrow trees (right).}
  \label{fig:decorated-tree-decomp}
\end{figure}

Let us assume first that $b>1$ (the case $b=1$ will be treated at the end of this subsection). 
\paragraph{$\boldsymbol{b}$-arrow trees.}
\begin{figure}[h]
  \centering
  \includegraphics[width=\textwidth]{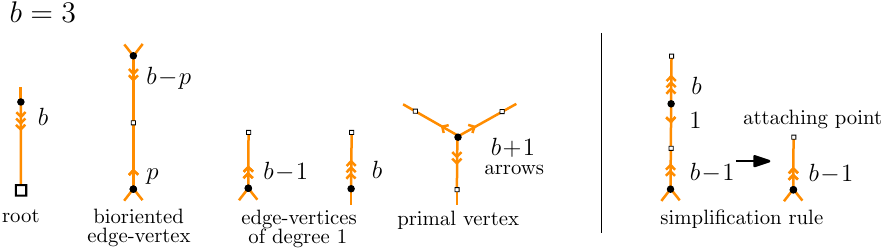}
  \caption{Environments of the various vertices in a $b$-arrow tree.
  For simplified $b$-arrow trees, non-root half-edges with $b$ arrows are removed as shown on the right.}
  \label{fig:arrowtreerules}
\end{figure}

Take a $b$-decorated tree corresponding to a $0$-slice, drawn on the derived map, and keep only the part of the tree drawn on the primal map, \emph{without the twigs}.
In other words, we keep only the primal half-edges which carry arrows and their incident vertices. This cuts the decorated
tree into a number of connected components, which are themselves plane trees built out of both bioriented edges and oriented half-edges 
(corresponding to the primal half-edges incident to a bent edge-vertex). We shall call \emph{$b$-arrow trees} these connected components. 
They are naturally planted by selecting as root the edge-vertex closest to the root of the $b$-decorated tree.
A $b$-arrow tree is then characterized as follows (see Figure \ref{fig:arrowtreerules})
\begin{itemize}
  \item[(i)] it is made of primal vertices of degree at least two and edge-vertices of degree one or two,
    connected by primal half-edges carrying between $1$ and $b$ arrows;
  \item[(ii)] around each primal vertex there is a total number of arrows equal to $(b+1)$;
  \item[(iii)] around each edge-vertex of degree two (still called bioriented edge-vertex) there is a total number of arrows equal to $b$;
  \item[(iv)] next to each edge-vertex of degree one there are either $b-1$ or $b$ arrows, the root vertex having $b$.
\end{itemize}

\paragraph{Simplified $\boldsymbol{b}$-arrow trees.}  
For the purposes of enumeration, it is useful to slightly simplify the above characterization thanks to the following remark.
Consider an edge-vertex of degree one with $b$ arrows which is \emph{not} the root of the $b$-arrow tree. 
To fulfill the condition (ii), its adjacent primal vertex necessarily has degree two and has one arrow on the other side.
Then it is in turn adjacent to a bioriented edge whose other half-edge carries $(b-1)$ arrows. 
We may thus at no cost remove this bivalent primal vertex and its two incident half-edges and keep only the
remnant half-edge with $(b-1)$ arrows (see Figure \ref{fig:arrowtreerules}-right and Figure \ref{fig:decorated-tree-decomp}-bottom right for an example).
Doing so for each half-edge with $b$ arrows different from the root of the $b$-arrow tree, 
we end up with a slightly simpler notion of what we hereafter call \emph{simplified $b$-arrow trees}, where point (iv) above is replaced by
\begin{itemize}
  \item[(iv')] next to each edge-vertex of degree one different from the root, there are $b-1$ arrows. These edge-vertices will be called \emph{attaching points}. Next to the root there are $b$ arrows.
\end{itemize} 
Note that each attaching point of a simplified $b$-arrow tree may be equally connected to a special or a labeled dual vertex in the $b$-decorated tree,
and that there is a unique way to make this connection:
this requires $2$ additional half-edges and vertices to undo the simplification in the case where the dual vertex is labeled.
\medskip 
\paragraph{Blossoming vertices.}
\begin{figure}[h]
  \centering
  \includegraphics[width=0.8\textwidth]{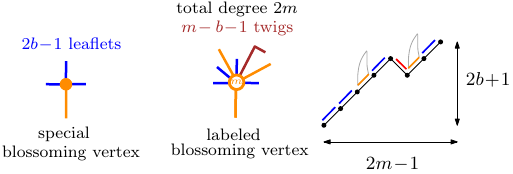}
  \caption{Representation of the order of decorations around blossoming vertices.
    Apart from their root, the special blossoming vertices are decorated with $2b-1$ leaflets,
   while labeled blossoming vertices with label $m$ are decorated with $m-b-1$ twigs
    and $m+b$ other decorations (interchangeably leaflets or attaching points).
  We represented on the left the proximity profile corresponding to the labeled blossoming vertex in the center, as defined
  in Figure~\ref{fig:step_III_dual}. Each leaflet corresponds to an up step associated with an empty slice,
  each attaching point to an up step with an associated $0$-slice and each twig to a down step associated 
  with a trivial slice. The example of labeled blossoming vertex shown here satisfies the tightness condition of  Proposition~\ref{prop:tightchar}.}
  \label{fig:blossoming-vertices}
\end{figure}
If we now keep only the part of the original decorated tree drawn on the dual map \emph{and the twigs},
all the dual vertices keep their degrees and we thus obtain a collection of \emph{blossoming vertices} (see Figure~\ref{fig:blossoming-vertices})
which are either special dual vertices of degree $2b$ or labeled dual vertices of degree $2m$ for some $m\geq b+1$.
We will call the blossoming vertices \emph{special} or \emph{labeled} accordingly.
Recall that a special blossoming vertex is decorated by $2b-1$ leaflets, and its incident half-edge
incident to its parent bent edge-vertex will be called the \emph{root} of this vertex. 
As for a labeled blossoming vertex of degree $2m$, its root is also the incident half-edge 
incident to its parent bent edge-vertex. The vertex is now decorated by $m-b-1$ twigs and a total of $m+b$ other half-edges, which are either leaflets
or \emph{attaching points}, which are the half-edges incident to its children bent edge-vertices in the decorated tree. 

Note that we use the same denominations ``root'' and ``attaching point''
for the (primal) $b$-arrow trees and the (dual) blossoming vertices.
In the decorated tree, the roots of blossoming vertices will be matched
with the attaching points of the arrow trees, and vice versa.

\bigskip

\begin{figure}
  \centering
  \includegraphics[width=.35\textwidth]{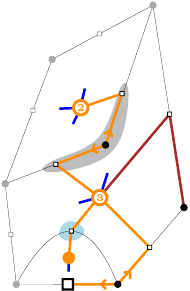}
  \caption{Example of a $1$-decorated tree.
    In gray and in light blue, the unique way to connect two labeled blossoming vertices (gray)
    or a labeled to a special blossoming vertex (light blue).
  }
  \label{fig:decomp-b-eq-1}
\end{figure}
\paragraph{The case $\boldsymbol{b=1}$.} In the case $b=1$, the discussion is different, since the connection between faces can be achieved via dual/dual edge-vertices.
We still have blossoming vertices of two types, special (with $2b-1=1$ leaflet) and labeled with $m-b-1=m-2$ twigs, which work similarly to the case $b>1$.
What would correspond to $1$-arrow trees is a collection of either degree-two primal vertices, each connected to two regular bent edge-vertex (when linking two labeled vertices),
or special dual/dual edge-vertices (when linking a special vertex to a labeled one).
In all cases, for each attaching point of a labeled vertex, there is a unique corresponding blossoming vertex,
and there is a unique way to make the connection
(with two or four half-edges, depending on whether the other blossoming vertex is special or labeled).
See Figure~\ref{fig:decomp-b-eq-1} for an example.

\subsection{Characterization of tightness}
\label{sec:tightchar}
So far, we did not impose that our slices be tight. A remarkable feature of the above decomposition of slices is that
the tightness of the slice is entirely characterized by simple constraints
on the decorations of the blossoming vertices in the associated decorated tree.

More precisely, recall that, in the context where there is no marked vertex, a slice
is tight whenever it has no leaf. Assume on the contrary that the slice contains a leaf $L$,
which is incident to a face $f$ of degree $2m$.
We know that $m > b$, as otherwise the contour of the face would contain
a cycle of size less than or equal to $2b-2$, which contradicts the $2b$-irreducibility.
The face $f$ will therefore give rise to a blossoming vertex labeled $m$,
which is built in a step (II).
With $c_s$ the corner of $f$ incident to $L$ (see Figure \ref{fig:step_III_dual} around vertex $L$, with $s=5$), we then know that:
\begin{itemize}
  \item[-] The proximity profile has $\ell_{s-1} = \ell_s + 1 = \ell_{s+1}$ ;
  \item[-] The slice $\boldsymbol{\sigma}_s$ is the trivial slice ;
  \item[-] The slice $\boldsymbol{\sigma}_{s+1}$ is the empty slice.
\end{itemize}
This latter property follows from the fact that
the leftmost geodesic from $L$ starts with the only edge leaving $L$,
which is also the base of the slice $\boldsymbol{\sigma}_{s+1}$.
Conversely, if we find a trivial slice followed by an empty slice around $f$, then
the identifications of the elementary slice boundaries shown on Figure~\ref{fig:step_III_dual}
create a leaf. For the blossoming vertex associated with $f$, 
this translates into a decoration where a twig follows immediately (in clockwise order) a leaflet.

We thus have the following characterization:
\begin{prop}[Characterization of tightness]
\label{prop:tightchar} 
A $2b$-irreducible slice (without marked vertices) is tight if and only if, at each blossoming vertex of the associated $b$-decorated tree,
the sequence of decorations read in clockwise order around this vertex from its root does not contain the pattern 
of a twig followed immediately by a leaflet. A decorated tree with this property will be said to be tight.
\end{prop}

\section{Enumeration}
\label{sec:enum}

Our goal is to obtain the expression~\eqref{eq:main-result} for the \emph{number} $\mathcal N_b(m_1, \ldots, m_n)$ of planar bipartite tight $2b$-irreducible maps
which can be constructed from a fixed number $n$ of labeled faces with prescribed even degrees $2m_1, \ldots, 2m_n$. 
From the bijection of Proposition~\ref{prop:annulbij} and the characterizations of Proposition \ref{prop:tightdecomp}, this can be done by enumerating, on the one hand, tight maps
with two faces of prescribed degrees (and marked vertices) with a control on the length of their unique cycle and, on the other hand,
sequences of tight $2b$-irreducible $0$-slices. Using the coding of $2b$-irreducible
$0$-slices by $b$-decorated trees, this latter enumeration translates into the counting of
sequences of $b$-decorated trees whose blossoming vertices are of a prescribed nature (i.e.\ special or
labeled with a prescribed label) and satisfy the tightness characterization described in Proposition~\ref{prop:tightchar}.
Let us first proceed to the counting of $b$-decorated trees, where we start by enumerating sequences of (simplified) $b$-arrow trees (Section \ref{sec:arrowenum})
and then combine them with blossoming vertices to obtain the desired sequences of $b$-decorated trees
(Section~\ref{sec:blossenum}).

\subsection{Sequences of simplified arrow trees}
\label{sec:arrowenum}
Let us start with some notation:
\begin{Def}$\boldsymbol{\alpha_{k,n}^{(b)}:}$
\label{def:alphakn}
Let $n$, $k$ and $b$ be non-negative integers with $b>1$.
We denote by $\alpha_{k,n}^{(b)}$ the number of ordered $(k+1)$-uples of simplified $b$-arrow
trees with a total of $n+1$ attaching points, one of them being distinguished in the
first arrow tree.
\end{Def}
When $b=1$, we did not use $b$-arrow trees for the decorated tree decomposition.
Still, we define:
\begin{equation}
  \alpha_{k,n}^{(1)} := \delta_{k,n}
  \label{eq:alpha-k-n-1-def}
\end{equation}
to account for the fact that there is a unique way to connect
children blossoming vertices to their parent blossoming vertex (see Figure~\ref{fig:decomp-b-eq-1} and the related discussion at the end of Section~\ref{sec:treecomponents}).
This connection is either via a dual/dual edge-vertex if the child is special,
or via a pair of bent edge-vertices otherwise.

\medskip
The goal of this section is to obtain an expression for $\alpha_{k,n}^{(b)}$ and to show that it is for $k \leq n$ a polynomial in $b$, of degree $2(n-k)$. Note that $\alpha_{k,n}^{(b)}$ vanishes for $k>n$ since a simplified $b$-arrow
tree has at least one attaching point.
Another related quantity of interest is the number $\mathcal U^{(b)}_{0,n}$ of simplified $b$-arrow trees with $n$ attaching points, for $n \geq 1$.
We have the relation \begin{equation}
  \label{eq:single-arrow-tree}
  \mathcal U^{(b)}_{0,n} = \frac{ \alpha_{0,n-1}^{(b)}}n ,
\end{equation}
as obtained upon forgetting the distinguished attaching point
in the first and unique tree counted by $\alpha_{0,n-1}^{(b)}$.

\begin{Def}
  \label{def:subtree}
  For $0 \leq p \leq b-1$, we define a \emph{simplified $b$-arrow tree of excess $p$}
  as a tree which follows the rules (i), (ii) and (iii) of simplified $b$-arrow trees,
  but instead of (iv') satisfies:
  \begin{itemize}
    \item[(iv'')] next to each edge-vertex of degree one different from the root, there are $b-1$ arrows. These edge-vertices will be called \emph{attaching points}.
      Next to the root there are $b-p$ arrows.
  \end{itemize} 
  We denote by $\mathcal U^{(b)}_{p,n}$ the number of simplified $b$-arrow trees of excess $p$ with $n$ attaching points.
  Note that this notation is consistent with our definition of $\mathcal U^{(b)}_{0,n}$ just above,
  since simplified $b$-arrow trees are nothing but simplified $b$-arrow trees of excess $0$.
  
  In the case $p=b-1$, when following the rules (i), (ii), (iii) and (iv''), the root edge (with one arrow) connects the root vertex to a primal vertex of degree at least three,
  so that there are at least two attaching points: as a consequence, $\mathcal U^{(b)}_{b-1,1}$ should a priori vanish.
  For convenience, we decide that the degenerate configuration of the tree consisting of a single edge-vertex, which serves both as root and as attaching point, is also considered as a simplified $b$-arrow tree of excess $p=b-1$. We therefore set accordingly
  \begin{equation}
    \mathcal U^{(b)}_{b-1,1} = 1.
    \label{eq:degenerate-subtree}
  \end{equation}
\end{Def}
\noindent See Figure~\ref{fig:subtree} for an illustration.
The variable $p$ acts as a catalytic variable, as we are eventually interested in $\mathcal U^{(b)}_{0,n}$.

\begin{figure}[h]
  \centering
  \includegraphics[width=0.6\textwidth]{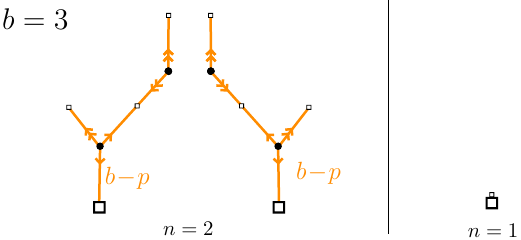}
  \caption{Examples of simplified $b$-arrow trees with excess $p=2$, when $b=3$. Left:
    the two possible cases with $n=2$ attaching points, leading to $\mathcal U^{(3)}_{2,2} = 2$. 
    Right: the degenerate configuration (which appears since $p=b-1$) with $n=1$ attaching point, leading to $\mathcal U^{(3)}_{2,1} = 1$.
  }
  \label{fig:subtree}
\end{figure}

\subsubsection{Elementary argument for the polynomiality in $b$}

Starting from the root of a simplified $b$-arrow tree with excess $p$,
we either are directly at an attaching point (if we are in the degenerate configuration, namely when $p=b-1$ and $n=1$),
or, climbing the tree, the root is connected to an inner (primal) vertex. If this vertex is of degree $2$,
then\footnote{Note that if $p=b-1$, this situation does not appear, since we would then have an edge bearing $b$ arrows not connected to the root, which is forbidden in simplified $b$-arrow trees.} its unique child is counted by $\mathcal U^{(b)}_{p+1,n}$.
We then climb the tree and cross all the vertices of degree $2$
until we reach an inner (primal) vertex with at least $2$ children, or an attaching point.
More precisely, we reach an attaching point if and only if $n=1$, which yields
\begin{equation}
  \mathcal U^{(b)}_{p,1} = 1\qquad\hbox{for $0 \leq p \leq b-1$.}
\end{equation}

Assume $n>1$, we then reach an inner vertex with $q \geq 2$ children.
We denote by $p_1, \ldots, p_q \in \{1,\ldots,b-1\}$ the excesses\footnote{Note that children subtrees cannot have excess $0$.} of the corresponding children subtrees,
and by $n_1, \ldots, n_q \geq 1$ their respective numbers of attaching points.
Clearly we have $n_1 + \cdots + n_q = n$.
We denote $s := p_1 + \cdots + p_q$. Then we have $p+1 \leq s \leq b$:
indeed, let us call $p_0$ the excess of the subtree just before we reach the vertex with $q$ children.
We have $p \leq p_0 \leq b-1$ since, by tree rules (ii) and (iv''),
the excess increases by $1$ at each crossing of an inner (primal) vertex of degree $2$.
Moreover, $(b-p_0) + p_1 + \cdots + p_q = b+1$ by the same rules, i.e.\ $s=p_0+1$.

As an example before we address the general case, let us first treat explicitly the case $n=2$.
The only integer composition of $n=2$ yields $q=2$, $n_1 = n_2 = 1$.
We then have
\begin{equation}
  \begin{split}
    \mathcal U^{(b)}_{p,2} &= \sum_{s = p+1}^{b} \sum_{\substack{p_1,p_2 \geq 1 \\ p_1 + p_2 = s}}\mathcal U^{(b)}_{p_1,1}\mathcal U^{(b)}_{p_2,1} 
    = \sum_{s = p+1}^{b} \sum_{p_1 = 1}^{s-1} 1 
    = \sum_{s = p+1}^{b} (s-1) \\
    &= \frac{b(b-1)}{2} - \frac{p(p-1)}{2} = \left[ \frac{m(m-1)}{2} \right]_p^b .
  \end{split}
  \label{eq:arrow-tree-n-2-q-2}
\end{equation}
From this, we deduce $\alpha_{0,1}^{(b)} = 2\mathcal U^{(b)}_{0,2} = b(b-1)$.
More generally we can write $\alpha_{n-1,n}^{(b)} = (2 + (n - 1))\mathcal U^{(b)}_{0,2}\left(\mathcal U^{(b)}_{0,1}\right)^{n-1}$.
Indeed, for $k=n-1$ we have exactly one arrow tree $\mathcal T_2$ with two attaching points, counted by $\mathcal U^{(b)}_{0,2}$, and $n-1$ other trees counted by $\mathcal U^{(b)}_{0,1}$.
This yields $n$ different sequences according to the position of $\mathcal T_2$,
each counted by $\mathcal U^{(b)}_{0,2}\left(\mathcal U^{(b)}_{0,1}\right)^{n-1}$.
As we distinguish an attaching point in the first tree, this yields
$2$ possibilities if $\mathcal T_2$ is in first position, and one possibility otherwise.
Altogether, with $\mathcal U^{(b)}_{0,1}=1$ and $\mathcal U^{(b)}_{0,2} = \frac{b(b-1)}{2}$ we have 
\begin{equation}
  \alpha_{n-1,n}^{(b)}=\frac{n+1}{2} b(b-1).
  \label{eq:alpha-n-1}
\end{equation}

The case $n=3$ is still doable by hand.
We either have $q=2$ or $q=3$.
When $q=3$, we have $n_1=n_2=n_3 = 1$, which yields after computations
\begin{equation}
  \begin{split}
  \sum_{s = p+1}^{b} \sum_{\substack{p_1,p_2,p_3 \geq 1 \\ p_1 + p_2 + p_3 = s}} \mathcal U^{(b)}_{p_1,1} \mathcal U^{(b)}_{p_2,1} \mathcal U^{(b)}_{p_3,1}
    &= \sum_{s = p+1}^{b} \sum_{p_1 = 1}^{s-1} \sum_{p_2 = 1}^{s-p_1-1}  \mathcal U^{(b)}_{p_1,1} \mathcal U^{(b)}_{p_2,1} \mathcal U^{(b)}_{s-p_1-p_2,1} \\
    &= \left[ \frac{m(m-1)(m-2)}{6} \right]_p^b .
  \label{eq:arrow-tree-n-3-q-3}
  \end{split}
\end{equation}

When $q=2$, we have the two symmetric cases $n_1=2,n_2=1$ and $n_1=1,n_2=2$.
The first (and the second, by symmetry) is counted by
\begin{equation}
  \begin{split}
  \sum_{s=p+1}^{b} \sum_{\substack{p_1,p_2 \geq 1 \\ p_1 + p_2 = s}}\mathcal U^{(b)}_{p_1,2}\mathcal U^{(b)}_{p_2,1} 
  = &\sum_{s = p+1}^{b} \sum_{p_1 = 1}^{s-1} \mathcal U^{(b)}_{p_1,2} \\
  = &\sum_{s = p+1}^{b} \sum_{p_1 = 1}^{s-1} \left[ \frac{m(m-1)}{2} \right]_{p_1}^b \\
  = &\left[ \frac{m(m-1)}{4}\left( b(b-1) - \frac{(m+1)(m-2)}{6}\right) \right]_p^b.
  \end{split}
  \label{eq:arrow-tree-n-3-q-2}
\end{equation}
Altogether, upon summing \eqref{eq:arrow-tree-n-3-q-3} and twice \eqref{eq:arrow-tree-n-3-q-2}, we get
\begin{equation}
  \begin{split}
  \mathcal U^{(b)}_{p,3} &= \left[ \frac{m(m-1)(m-2)}{6} + 2\frac{m(m-1)}{4}\left( b(b-1) - \frac{(m+1)(m-2)}{6}\right)  \right]_p^b \\
  &= \left[ \frac{m(m-1)}{2} \left( b(b-1) - \frac{(m-1)(m-2)}{6}\right)  \right]_p^b.
  \end{split}
\end{equation}
We may then, from this formula and \eqref{eq:arrow-tree-n-2-q-2}, derive the explicit expression:
\begin{equation}
  \alpha_{n-2,n}^{(b)} = \frac{n+1}{6} \left (\frac{3n+4}{4} b + 1 \right )b(b-1)^2 .
  \label{eq:alpha-n-2}
\end{equation}
This follows from $\alpha_{n-2,n}^{(b)}=(3+(n-2))\mathcal U^{(b)}_{0,3}\left( \mathcal U^{(b)}_{0,1} \right)^{n-2} \! + (2(n-2) + \binom{n-2}2)\!\left( \mathcal U^{(b)}_{0,2} \right)^2\!\left( \mathcal U^{(b)}_{0,1} \right)^{n-3}$ obtained similarly to the derivation of Equation~\eqref{eq:alpha-n-1}. This is also a particular case of the general Equation~\eqref{eq:alpha-k-n-direct} which we will see just below.

\medskip

Let us now discuss the case of general $n$. We have the following:
\begin{prop}
  For $b>1$, $n \geq 1$, $p \in \{0,\ldots b-1\}$, the quantity $\mathcal U^{(b)}_{p,n}$ is a polynomial
  in $b$ and $p$ of total degree $2(n-1)$, with a non-zero coefficient for $b^{2(n-1)}$.
  \label{prop:U-polynomial}
\end{prop}

\begin{proof}
In the general case,
we sum over all possible values of $q \geq 2$, over $n_1, \ldots, n_q$ with sum $n$, over $s \geq p+1$, and over $p_1, \ldots, p_q$ with sum $s$.
The trick is that, for fixed $n$, we have a finite number of configurations of $q \in \left\{ 2, \ldots, n \right\}$ and $n_1, \ldots, n_q \geq 1$ summing to $n$.
We then have the recurrence formula, valid for $0 \leq p \leq b-1$ and $n \geq 2$:
\begin{equation}
  \mathcal U^{(b)}_{p,n} = \sum_{q = 2}^{n} \sum_{\substack{n_1, \ldots, n_q \geq 1 \\ n_1 + \cdots + n_q = n}}
                     \sum_{s=p+1}^{b} \sum_{\substack{p_1, \ldots, p_q \geq 1 \\ p_1 + \cdots + p_q = s}}
                       \mathcal U^{(b)}_{p_1,n_1} \cdots \mathcal U^{(b)}_{p_q,n_q}
  \label{eq:arrow-tree-recurrence}
\end{equation}
which allows us to prove Proposition~\ref{prop:U-polynomial} by recurrence.
Indeed, the sum over $s$ and over the simplex of the $p_1, \ldots, p_q$ is a discrete integration of dimension $q$ of the polynomial $\mathcal U^{(b)}_{p_1,n_1} \cdots \mathcal U^{(b)}_{p_q,n_q}$,
which yields a polynomial of degree $2(n_1-1) + \cdots + 2(n_q-1) + q = 2n- q$ in $b$ and $p$.
The dominant terms in the recurrence formula then come from the terms for $q=2$, with a total degree of $2(n - 1)$,
which, with initialization $\mathcal U^{(b)}_{p,1}=1$, concludes the recurrence for the total degree.
Setting $p=0$ yields the degree $2(n-1)$ in $b$ alone.
\end{proof}

This leads to the following:
\begin{cor}
  \label{cor:alpha-polynomial}
  For $k \leq n$, $\alpha_{k,n}^{(b)}$ is a polynomial in $b$ of degree $2(n-k)$.
\end{cor}
\begin{proof}
We may retrieve, in the general case, the value of $\alpha_{k,n}^{(b)}$ by the formula:
\begin{equation}
  \alpha_{k,n}^{(b)} = \sum_{\substack{n_1, \ldots, n_{k+1} \geq 1\\ n_1 + \cdots + n_{k+1} = n+1}} n_1 \mathcal U^{(b)}_{0,n_1} \cdots \mathcal U^{(b)}_{0,n_{k+1}} .
  \label{eq:alpha-k-n-direct}
\end{equation}
arising from Definition~\ref{def:alphakn}.
In particular, since $\mathcal U^{(b)}_{0,n_i}$ is of degree $2(n_i - 1)$ in $b$, $\alpha_{k,n}^{(b)}$ is a polynomial in $b$ of degree $2(n_1 - 1) + \cdots + 2(n_{k+1} -1) = 2 (n+1) - 2(k+1) = 2(n-k)$.
\end{proof}

\subsubsection{Direct formula via Lagrange inversion}
\label{sec:laginv}

Besides the above purely combinatorial approach, we may obtain a slightly more explicit direct (non-recursive) formula for $\alpha_{k,n}^{(b)}$
with an analytic approach based on a Lagrange inversion.
More precisely, let us establish the expression~\eqref{eq:alpha-k-n} for $\alpha_{k,n}^{(b)}$, namely:
\begin{prop}
  \label{prop:arrowtreeformula}
  For $k,n$ non-negative integers and $b\geq 1$, we have
  \begin{equation}
    \alpha_{k,n}^{(b)} = [u^{n-k}] \frac{1}{\left( 1+ \sum\limits_{j=2}^b \frac1b \binom{b}j \binom{b}{j-1} (-u)^{j-1} \right)^{n+1}}.
    \label{eq:arrowtreeformula}
  \end{equation}
\end{prop}

\begin{proof} For $b=1$ this indeed yields $\alpha_{k,n}^{(1)} = \delta_{k,n}$ which is the desired value defined in Equation~\eqref{eq:alpha-k-n-1-def}. Assume now $b>1$ and 
let $U_0(z)$ be the generating function of simplified $b$-arrow trees
counted with a weight $z$ per attaching point. 
By Definition~\ref{def:alphakn}, we have
\begin{equation}
  \label{eq:arrowtreedef}
  \alpha_{k,n}^{(b)} = [z^n] U_0'(z) U_0(z)^k = \frac{n+1}{k+1} [z^{n+1}] U_0(z)^{k+1}.
\end{equation}
More generally, for $0 \leq p \leq b-1$, we introduce the generating function $U_p(z)$ of
simplified $b$-arrow trees of excess $p$,
still counted with a weight $z$ per attaching point.
We may write
\begin{equation}
  U_p(z) = \sum_{n \geq 1} \mathcal U^{(b)}_{p,n}z^n, \qquad 0 \leq p \leq b-1.
  \label{eq:u-GF-definition}
\end{equation}

We may now obtain the formula:
\begin{equation}
  U_p(z) = \sum_{q=1}^{p+1} \sum_{\substack{p_1, \ldots, p_q \geq 1\\p_1 + \cdots + p_q = p+1}} \prod_{j=1}^{q}U_{p_j}(z),
  \qquad 0 \leq p \leq b-1,
  \label{eq:rec_2p-slice_genfunc}
\end{equation}
with the convention $U_b(z) := z$.
This values ensures that, when $p=b-1$, the $q=1$ term of the sum, which is equal to $U_b(z)$, gets the proper value $z$, consistent with~\eqref{eq:degenerate-subtree}.
Equation~\eqref{eq:rec_2p-slice_genfunc} simply expresses the fact that, in a non-degenerate simplified $b$-arrow tree of excess $p$, the root vertex is connected to a primal vertex which has a number $q\geq1$ of other neighbours which are edge-vertices.
Removing this primal vertex and its incident primal half-edges, the tree is split into $q$ rooted subtrees, which are simplified $b$-arrow trees with respective excesses denoted $p_1, \ldots, p_q$.
From rules (ii), (iii) and (iv''), we get the constraint $p_1 + \cdots + p_q = p+1$.

Note the similarity of this decomposition with the case (I) of the decomposition of $b$-irreducible $p$-slices:
indeed, for $0 \leq p \leq b-1$, $U_p(z)$ also counts $b$-irreducible $p$-slices with all inner faces of degree $2b$, each weighted by $z$.

\begin{rem}
  Note that Equation~\eqref{eq:rec_2p-slice_genfunc} can be obtained from Equation~\eqref{eq:arrow-tree-recurrence}
  by isolating the term $s=p+1$ in the latter and identifying the rest of the sum as $\mathcal U^{(b)}_{p+1,n}$, which yields:
  \begin{equation}
    \mathcal U^{(b)}_{p,n} = \mathcal U^{(b)}_{p+1,n} + \sum_{q = 2}^{n} \sum_{\substack{n_1, \ldots, n_q \geq 1 \\ n_1 + \cdots + n_q = n}}
                     \sum_{\substack{p_1, \ldots, p_q \geq 1 \\ p_1 + \cdots + p_q = p+1}}
                     \mathcal U^{(b)}_{p_1,n_1} \cdots \mathcal U^{(b)}_{p_q,n_q}, \qquad 0 \leq p \leq b-1
    \label{eq:arrow-tree-recurrence-unrolled}
  \end{equation}
  with the convention that $\mathcal U^{(b)}_{b,n} := \delta_{n,1}$ consistent with the convention $U_b(z) = z$.
  Then, translated in generating series, this yields
  \begin{equation}
    U_p(z) = U_{p+1}(z) + \sum_{q=2}^{p+1} \sum_{\substack{p_1, \ldots, p_q \geq 1\\p_1 + \cdots + p_q = p+1}} \prod_{j=1}^{q}U_{p_j}(z), \qquad 0 \leq p \leq b-1
    \label{eq:rec_2p-slice_genfunc-unrolled}
  \end{equation}
  which is equivalent to Equation~\eqref{eq:rec_2p-slice_genfunc}.
\end{rem}

In  its equivalent form \eqref{eq:rec_2p-slice_genfunc-unrolled},
we see that the system \eqref{eq:rec_2p-slice_genfunc} is triangular,
as it may be rewritten
\begin{equation}
  U_{p+1}(z) = U_p(z) - \sum_{q=2}^{p+1} \sum_{\substack{p_1, \ldots, p_q \geq 1\\p_1 + \cdots + p_q = p+1}} \prod_{j=1}^{q}U_{p_j}(z)
  \label{eq:rec_2p-slice_genfunc-unrolled-bis}
\end{equation}
for $0 \leq p \leq b-1$, where the sum in the right-hand side involves
only $U_1(z),\ldots,U_p(z)$. It follows that $U_p(z)$ is a polynomial
in $U_0(z)$ for all $p$. We may express it explicitly via the
following trick, borrowed from~\cite[Section~5.4]{irredmaps}: let us
\emph{define} $U_{p+1}(z)$ recursively for all $p \geq b$ via the
relation~\eqref{eq:rec_2p-slice_genfunc-unrolled-bis}. Note that, in
this relation, we may take the sum over $q$ from $2$ to $\infty$,
since the terms $q>p+1$ give no contribution. Then, in terms of the
generating function $U(t,z) := \sum_{p \geq 1} U_p(z) t^p$, the
relation yields
\begin{equation}
  U(t,z) = t (U(t,z) + U_0(z)) - \sum_{q=2}^{\infty} U(t,z)^q
\end{equation}
which may be rewritten as
\begin{equation}
  t = \frac{U(t,z)}{(1-U(t,z))(U(t,z)+U_0(z))}.
\end{equation}
Using the Lagrange inversion formula, we obtain, for all $p \geq 1$,
\begin{equation}
    U_p(z) = [t^p] U(t,z) = \frac 1 p [u^{p-1}]\left( (1-u)(u+U_0(z)) \right)^p = h_p(U_0(z))
\end{equation}
where
\begin{equation}
  \label{eq:hpdef}
  h_p(u) := \sum_{j=1}^{p} \frac{(-1)^{j-1}}p \binom{p}{j} \binom{p}{j-1} u^j
\end{equation}
is a polynomial in $u$, with zero constant term and with linear term
$u$. 
Recalling that $U_b(z)=z$, we
find that $U_0(z)$ is algebraic and determined implicitly by
\begin{equation}
  \label{eq:U0implicit}
  z = h_b(U_0(z)).
\end{equation}
Applying the Lagrange-Bürmann inversion formula to~\eqref{eq:arrowtreedef}, we get
\begin{equation}
  \label{eq:alphalagburr}
  \alpha_{k,n}^{(b)} = \frac{n+1}{k+1} [z^{n+1}] U_0(z)^{k+1} = [u^{n-k}] \left( \frac u {h_b(u)} \right)^{n+1}
\end{equation}
from which~\eqref{eq:arrowtreeformula}, hence~\eqref{eq:alpha-k-n},
follows immediately.
\end{proof}

Pushing further the above computations, we may arrive at another
expression for $\alpha_{k,n}^{(b)}$, which has the interest of being
manifestly polynomial in $b$:
\begin{prop}
  \label{prop:alphapolexp}
  For $k,n$ non-negative integers, we have
  \begin{equation}
    \label{eq:alphapolexp}
    \alpha_{k,n}^{(b)} = \sum_{s=0}^\infty (-1)^{n-k+s} \binom{n+s}n
      \sum_{\substack{\ell_1,\ldots,\ell_s \geq 1 \\ \ell_1+\cdots+\ell_s=n-k}} r_{\ell_1}(b) \cdots r_{\ell_s}(b)
  \end{equation}
  where
  \begin{equation}
    \label{eq:rldef}
    r_\ell(b) := \frac1b \binom{b}{\ell+1} \binom{b}{\ell} =
    \frac{1}{\ell! (\ell+1)!} \prod_{i=1}^{\ell}(b-i+1)(b-i)
  \end{equation}
  is a polynomial of degree $2\ell$ in $b$. As a consequence,
  $\alpha_{k,n}^{(b)}$ is a polynomial of degree $2(n-k)$ in $b$ for
  $k \leq n$, and vanishes for $k>n$.
\end{prop}

Let us remark that, in~\eqref{eq:alphapolexp}, the term $s=0$
contributes only for $k=n$, and the rightmost sum is then equal to $1$
as it involves a single term corresponding to the empty sequence: this
yields $\alpha_{n,n}^{(b)}=1$ as wanted. For $k>n$, the rightmost sum
vanishes for all $s$, since no sequence satisfies the wanted
condition: this yields $\alpha_{k,n}^{(b)}=0$ as expected. Note finally
that $r_\ell(0)=r_\ell(1)=0$ for all $\ell \geq 1$, and hence
$\alpha_{k,n}^{(0)}=\alpha_{k,n}^{(1)}=\delta_{k,n}$ as wanted.

\begin{proof}[Proof of Proposition~\ref{prop:alphapolexp}]
  Observe that, in~\eqref{eq:hpdef}, we may replace the upper bound of
  the sum by $\infty$ since all terms $j>p$ have a zero
  contribution. Replacing $p$ by $b$, doing the change of variable
  $j=\ell+1$, and putting the first term apart, this allows to rewrite
  \begin{equation}
    \label{eq:hbalt}
    h_b(u) = u + \sum_{\ell=1}^\infty (-1)^\ell r_\ell(b) u^{\ell+1}
  \end{equation}
  with $r_\ell(b)$ as in the proposition. Plugging this expression in
  the right-hand side of~\eqref{eq:alphalagburr}, we get
  \begin{equation}
    \begin{split}
      \alpha_{k,n}^{(b)}  &= [u^{n-k}] \frac1{\left(1 + \sum_{\ell=1}^\infty (-1)^\ell r_\ell(b) u^{\ell} \right)^{n+1}} \\
                          &= [u^{n-k}] \sum_{s=0}^\infty (-1)^s \binom{n+s}n \left( \sum_{\ell=1}^\infty (-1)^\ell r_\ell(b) u^{\ell} \right)^s
    \end{split}
  \end{equation}
  which gives the wanted formula~\eqref{eq:alphapolexp}.
\end{proof}

We list for bookkeeping purposes the following simple values for $k=n,n-1,n-2, n-3$:
\begin{itemize}
  \item $\alpha_{n,n}^{(b)} = 1$ for all $n$ ;
  \item $\alpha_{n-1,n}^{(b)} = \frac{n+1}{2} b(b-1)$ for $n \geq 1$, see Equation~\eqref{eq:alpha-n-1};
  \item $\alpha_{n-2,n}^{(b)} = \frac{n+1}{6} \left(\frac{3n+4}{4} b + 1\right)b(b-1)^2$ for $n \geq 2$, see Equation~\eqref{eq:alpha-n-2};
  \item $\alpha_{n-3,n}^{(b)} = \frac{n+1}{12} \left( \frac{3n^2 + 9n + 7}{12} b^3 
    - \frac{3n^2-3n-11}{12}b^2 - \frac{3n+2}{3}b - 1\right)b(b-1)^2$ for $n \geq 3$.
\end{itemize}

\subsection{Sequences of decorated trees}
\label{sec:blossenum}

The purpose of this section is to establish the following:
\begin{prop}
\label{prop:Fbkm}
  For integers $n,b\geq 1$, $k\geq 0$ and $m_1,\ldots,m_n \geq b$, the number of $(k+1)$-tuples of tight $b$-decorated trees having $n$ blossoming vertices of degrees $2m_1,\ldots,2m_n$, where the first tree contains the blossoming vertex of degree $m_1$, is equal to
  \begin{equation}
  \label{eq:Fbkm}
    F^{(b)}_k(m_1,\ldots,m_n) :=
    (n-1)! \sum_{k_1,\ldots,k_n \geq 0} q^{(b)}_{k_1}(m_1) \cdots q^{(b)}_{k_n}(m_n) \alpha^{(b)}_{k+k_1+\cdots+k_n,n-1}
  \end{equation}
  with $\alpha^{(b)}_{k,n}$ as in~\eqref{eq:alpha-k-n} and, recalling Equation~\eqref{eq:qkb},
  \begin{equation}
    \label{eq:qkbdef}
    q^{(b)}_{k}(m) := \binom{m+b}{k} \binom{m-b-1+k}{k} = \frac1{(k!)^2} \prod_{i=0}^{k-1} \left( m^2 - (b-i)^2 \right).
  \end{equation}
\end{prop}

By the correspondence between decorated trees and slices, we
immediately obtain the following:

\begin{cor}
  \label{cor:Fbkm}
  For integers $n,b\geq 1$, $k\geq 0$ and integers
  $m_1,\ldots,m_n \geq b$, the quantity $F^{(b)}_k(m_1,\ldots,m_n)$
  in~\eqref{eq:Fbkm} is the number of $(k+1)$-tuples
  $(\mathbf{s}_1,\ldots,\mathbf{s}_{k+1})$ of tight $2b$-irreducible
  $0$-slices such that there is a bijection between $\{1,\ldots,n\}$
  and the union of the sets of the inner faces of
  $\mathbf{s}_1,\ldots,\mathbf{s}_{k+1}$, such that
  each $j=1,\ldots,n$ is mapped to a face of degree $2m_j$, and $1$ is
  mapped to an inner face of $\mathbf{s}_1$. In particular, for $k=0$,
  the number of tight $2b$-irreducible $0$-slices with $n$ inner faces
  of degrees $2m_1,\ldots,2m_n$ is equal to
  \begin{equation}
    \label{eq:Fb0m}
    F^{(b)}_0(m_1,\ldots,m_n)  = (n-1)! \sum_{k_1,\ldots,k_n \geq 0} q^{(b)}_{k_1}(m_1) \cdots q^{(b)}_{k_n}(m_n) \alpha^{(b)}_{k_1+\cdots+k_n,n-1}.
  \end{equation}
\end{cor}
Recall that  $\alpha_{k,n}^{(b)}$ vanishes for $k>n$ hence the sums in \eqref{eq:Fbkm}
and \eqref{eq:Fb0m} are \emph{finite} sums.
\medskip

The first step in the proof of Proposition \ref{prop:Fbkm} consists in observing that the quantity $q^{(b)}_{k}(m)$, which is a polynomial in $m^2$ and $b$, counts the number of possible decorations around a blossoming vertex of degree $2m$ with $k$ attaching points satisfying the tightness condition of Proposition \ref{prop:tightchar}. Indeed, for $m=b$ we have $q^{(b)}_{k}(b)=\delta_{k,0}$ as wanted for a special vertex which by definition has $0$ attaching point. For $m>b$, each decoration is coded by a word of length $2m-1$ over the alphabet $\{A,L,T\}$ (where these letters stand for attaching point, leaflet and twig respectively) with $k$ occurrences of $A$, $m+b-k$ occurrences of $L$ and $m-b-1$ occurrences of $T$, and no occurrence of the pattern $TL$. Such a word has the form
\begin{equation}
  L^{i_0} T^{j_0} A L^{i_1} T^{j_1} A \cdots A L^{i_k} T^{j_k} 
\end{equation}
where $i_0,i_1,\ldots,i_k$ are non-negative integers summing to $m+b-k$ and $j_0,j_1,\ldots,j_k$ are non-negative integers summing to $m-b-1$. There are $\binom{m+b}{k}$ choices for the former and $\binom{m-b+k-1}{k}$ for the latter, leading to the expression \eqref{eq:qkbdef}.

\medskip
The second step of the proof consists of the following:
\begin{lem}
  Let $n,b,k,m_1,\ldots,m_n$ be as in Proposition \ref{prop:Fbkm} and fix non-negative integers $k_1,\dots,k_n$. Then,
  the number of tuples of decorated trees as in Proposition \ref{prop:Fbkm} where we add the requirement that, for each $i=1,\ldots,n$, the $i$-th blossoming vertex has exactly $k_i$ attaching points, is equal to
$(n-1)! q^{(b)}_{k_1}(m_1) \cdots q^{(b)}_{k_n}(m_n) \alpha^{(b)}_{k+k_1+\cdots+k_n,n-1}$.
\end{lem}

\begin{proof}
This results from general considerations on the enumeration of plane forests with labeled vertices, similar to those developed in \cite[Appendix A]{polytightmaps}.

Precisely, we first observe that for any $k_0\geq 0$, the quantity $(n-1)!\alpha^{(b)}_{k_0,n-1}$ is the number of $(k_0 + 1)$-tuples of simplified $b$-arrow trees with a total of $n$ attaching points, that are \emph{numbered} from $1$ to $n$, the one numbered $1$ being in the first arrow tree.
Indeed, in $\alpha^{(b)}_{k_0,n-1}$ there is a unique distinguished attaching point which is in the first tree, which we number $1$, and we have $(n-1)!$ ways to number the other attaching points.
Now, we take $k_0= k+k_1+\cdots+k_n$, and for each $i=1,\ldots,n$ we choose a decoration for the $i$-th blossoming vertex in one of the $q^{(b)}_{k_i}(m_i)$ ways, and attach the root of that blossoming vertex to the arrow tree attaching point numbered $i$.
At this stage, we no longer have free attaching points incident to arrow trees, but we still have $k_1+\cdots+k_n=k_0-k$ free attaching points incident to blossoming vertices, while the number of trees is still $k_0+1$.
There is then a canonical way to assemble these trees into a sequence of $(k_0+1) - (k_0-k) = k+1$ decorated trees with the first tree containing the first blossoming vertex.
Indeed, we represent each of the $k_0+1$ tree by a sequence formed by a simple up step followed by a number of down steps equal to its number of free attaching points.
This yields a sequence of steps with $k_0+1$ up steps and $k_0 - k$ down steps.
We write this sequence cyclically, and connect each down step with the closest available following up step with non-crossing arches;
in the original trees, this corresponds to connecting the free attaching points to some of the roots.
This yields a cyclic sequence of $k+1$ decorated trees, which we break into a linear sequence by demanding that the first blossoming vertex be in the first tree. This is precisely a tuple of decorated trees of the wanted type.
\end{proof}

From this lemma, Proposition~\ref{prop:Fbkm} is deduced by summing
over all possible $k_1,\ldots,k_n$.

\subsection{Two-face maps with marked vertices}
\label{sec:unicycenum}

Having enumerated $(k+1)$-tuples
  of tight $2b$-irreducible
  $0$-slices in the previous section, we are now left with the counting
  of maps $\mathbf{m}_{12}$ with two faces of prescribed degrees, 
  and $(k+1)$ marked vertices.
More precisely, it is the purpose of the present section to establish the following:
\begin{prop}
  \label{prop:pcmm}
  Let $k,c,m_1,m_2$ be non-negative integers with $m_1,m_2>c$. Then,
  the number $p^{(c)}_k(m_1,m_2)$ of tight planar maps with exactly
  two (labeled) faces of respective degrees $2m_1$ and $2m_2$, with
  $(k+1)$ among their vertices marked, one of them being distinguished,
  and with their unique cycle of length at least $2(c+1)$, is equal to 
  \begin{equation}
    \label{eq:pcmm}
    p^{(c)}_k(m_1,m_2) = \sum_{\substack{k_1,k_2 \geq 0 \\ k_1+k_2=k}} p^{(c)}_{k_1}(m_1)q^{(c)}_{k_2}(m_2)
  \end{equation}
  with $q^{(c)}_{k_2}(m_2)$ as in~\eqref{eq:qkb} or \eqref{eq:qkbdef} with $b,k,m \to c,k_2,m_2$, namely
  \begin{equation}
    \label{eq:qkcdef}
        q^{(c)}_{k_2}(m_2) = \binom{m_2+c}{k_2} \binom{m_2-c-1+k_2}{k_2} = \frac1{(k_2!)^2} \prod_{i=0}^{k_2-1} \left( m_2^2 - (c-i)^2 \right)
  \end{equation}
  and $p^{(c)}_{k_1}(m_1)$ as in~\eqref{eq:pkb}, namely
  \begin{equation}
    \label{eq:pkcdef}
    p^{(c)}_{k_1}(m_1) := \binom{m_1-c-1}{k_1} \binom{m_1+c+k_1}{k_1} = \frac1{(k_1!)^2} \prod_{i=1}^{k_1} \left( m_1^2 - (c+i)^2 \right).
  \end{equation}
\end{prop}

Note that, borrowing from \cite{polytightmaps} the notation
\begin{equation}
  p_{k,e}(m) := \binom{m+\frac{e}{2} -1}{k} \binom{m - \frac{e}{2} +k}{k} = \frac{1}{(k!)^2} \prod_{i=1}^{k}(m^2 - (i-\frac{e}{2})^2)
\end{equation}
for $k \geq 0, e \in \mathbb Z$, we have
$p^{(c)}_{k}(m) = p_{k, -2c}(m)$ and
$q^{(c)}_{k}(m) = p_{k, 2c+2}(m)$.

We emphasize that $p^{(c)}_k(m_1,m_2)$ counts two-face planar maps
whose unique cycle has length \emph{at least} $2(c+1)$. A first idea
would be to count maps for which the unique cycle has a \emph{fixed}
length, and then perform a summation. While this may be done, see
Remark~\ref{rem:fixedcyclelength} below, it is actually not the most
direct route to the expression~\eqref{eq:pcmm}. Instead, we adapt the
approach of~\cite[Sections~4.2 and 5.2]{polytightmaps} which relies on
a bijection between two-face maps and pairs of sequences of
trees with possibly different lengths.

\label{sec:gluingtwoforests}
\begin{figure}[t]
  \centering
  \includegraphics[width=0.9\textwidth]{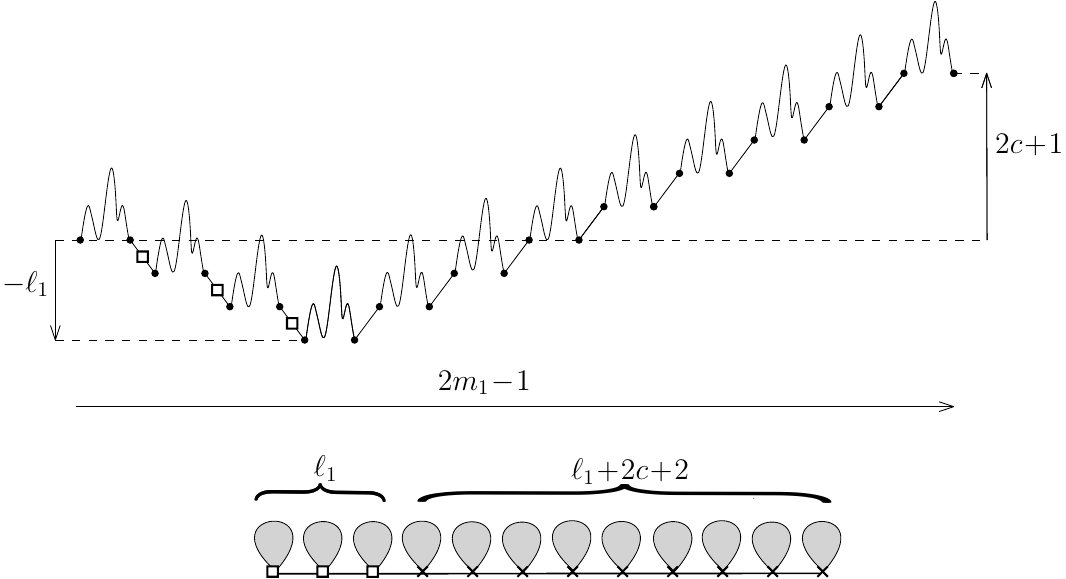}
  \caption{Top: a schematic picture of a word of the form
    \eqref{eq:UMD}, viewed equivalently as a lattice path with height
    difference $2c+1$ (here $c=2$). Calling $-\ell_1$ the minimum
    height (here $\ell_1=3$), we may decompose this path into
    $2\ell_1+2c+2$ blocks (Dyck paths) separated by $\ell_1$ down
    steps, which are \emph{markable} (as indicated by squares) and
    $\ell_1+2c+1$ up steps. Bottom: alternatively, this codes for a
    sequence $\mathcal S_1$ of $2\ell_1+2c+2$ rooted plane trees
    (displayed as gray blobs connected by a \emph{spine} of
    edges). The roots of the first $\ell_1$ trees can be marked or not
    (as indicated by squares).  The roots of the last $\ell_1+2c+2$
    trees are not marked (as indicated by crosses).  }
  \label{fig:pathstoforests1}
\end{figure}
We first need to enumerate sequences of rooted plane
trees, with some vertices marked, such that every leaf is marked.
Such sequences are best encoded by concatenating Dyck paths with marked
steps, where the marking of the leaves is enforced by forbidding a
certain pattern.  The sequences of such Dyck paths are themselves in
one-to-one correspondence with appropriate words over a three-letter
alphabet avoiding some pattern.

More precisely, take integers $c,k_1\geq 0$ and $m_1\geq c+1$,
and consider words of length $2m_1-1$ over the alphabet $\{M,D,U\}$
(where these letters stand for marked-down, down and up respectively)
with $k_1$ occurrences of $M$, $m_1-c-1-k_1$ occurrences of $D$ and $m_1+c$ occurrences of $U$,
and with no occurrence of the pattern $UD$. Any such word has the form
\begin{equation}
\label{eq:UMD}
  D^{i_0} U^{j_0} M D^{i_1} U^{j_1} M \cdots M D^{i_{k_1}} U^{j_{k_1}}
\end{equation}
where $i_0,i_1,\ldots,i_{k_1}$ are non-negative integers summing to $m_1-c-1-k_1$
and $j_0,j_1,\ldots,j_{k_1}$ are non-negative integers summing to $m_1+c$. 
There are $\binom{m_1-c-1}{k_1}$ choices for the former and $\binom{m_1+c+k_1}{k_1}$ for the latter,
leading to the total number of words $p^{(c)}_{k_1}(m_1)$ as in~\eqref{eq:pkcdef}.
Note that $p^{(c)}_{k_1}(m_1)=0$ for $k_1\geq m_1-c$ as it should. 

By interpreting the letter $U$ as coding for an elementary up step $(1,1)$,
and the letter $D$ (respectively $M$)
as coding for an elementary down step (respectively a \emph{marked} elementary down step) $(1,-1)$,
each word codes for a directed lattice path in $\Z^2$ starting at $(0,0)$,
of total length (number of steps) equal to $(m_1+c)+(m_1-c-1-k_1)+k_1=2m_1-1$,
and height difference (final ordinate) equal to $(m_1+c)-(m_1-c-1-k_1)-k_1=2c+1$.
This path is equipped with a total of $k_1$ markings on the down steps.
Calling $-\ell_1$ the minimal height of this path
(see Figure~\ref{fig:pathstoforests1}), with $\ell_1\geq 0$ by construction, 
we can decompose the path into $2\ell_1+2c+2$ blocks separated by the $\ell_1$ down steps
(that can be marked or not) which correspond to the first passage
at height $h$ for $h=-1,\ldots,-\ell_1$ followed by the $\ell_1+2c+1$ up steps\footnote{
that cannot be marked, as we only mark down steps.}
which correspond to the last passage at height $h$ for $h=-\ell_1,\ldots,2c$. 
Each block is a Dyck path made of an equal number of up and down steps which, as is well-known, is the \emph{contour path} of a rooted plane tree: recall that the contour path of a rooted plane tree with a total of $m$ non-root vertices is the path of length $2m$ obtained by going clockwise around the tree from its root and recording the distance to the root of the successive encountered corners. Each non-root vertex in the tree is in correspondence with a down step along the Dyck path (following the last passage at this vertex along the contour) and we may therefore transfer the markings of the down steps to their associated non-root vertices. 
The absence of the $UD$ pattern then guarantees that \emph{all the (non-root) leaves of the tree are marked}. As for the
root of the tree, we mark it if the block at hand is followed by a marked down step, which is possible only for the
$\ell_1$ first blocks in the decomposition.  Altogether, we arrive at the following:
\begin{prop}
  \label{prop:def_F1}
For fixed integers $c,k_1\geq 0$ and $m_1\geq c+1$, the number of sequences $\mathcal S_1$ made of $2\ell_1+2c+2$ rooted plane trees with a total of $m_1-(c+1)-\ell_1$ internal edges, for some (unfixed) value $\ell_1$
with $0\leq \ell_1 \leq m_1-(c+1)$ , with a total of $k_1$ marked vertices and such that all the (non-root) leaves of the trees are marked and the roots of the last $\ell_1+2c+2$ plane trees are not marked is given by 
$ p^{(c)}_{k_1}(m_1)$ as in \eqref{eq:pkcdef}.
\end{prop} 

\begin{figure}[t]
  \centering
  \includegraphics[width=.7\textwidth]{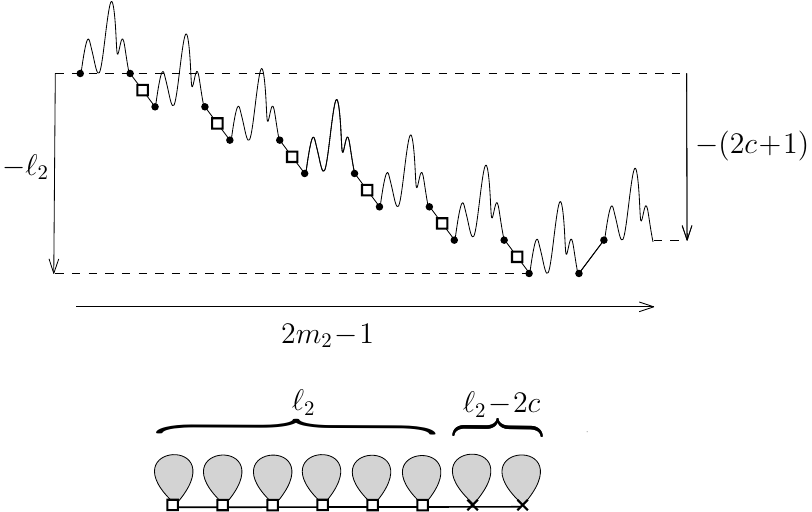}
  \caption{Top: a lattice path
  with height difference $-(2c+1)$ having minimal height $-\ell_2$ (with $\ell_2\geq 2c+1$, here $c=2$ and $\ell_2=6$). We may decompose this path as in Figure \ref{fig:pathstoforests1}, creating now $2\ell_2-2c$ blocks (Dyck paths). Bottom: alternatively, this codes for
 a sequence $\mathcal S_2$ of $2\ell_2-2c$ rooted plane trees (displayed as gray blobs connected by a spine of edges), where the roots of the first $\ell_2$
 trees are markable, while the roots of the last $\ell_2-2c$ are not.}
  \label{fig:pathstoforests2}
\end{figure}
Take now integers $c,k_2\geq 0$ and $m_2\geq c+1$ and consider words of length $2m_2-1$ over the alphabet $\{M,D,U\}$ with $k_2$ occurrences of $M$, $m_2+c-k_2$ occurrences of $D$ and $m_2-c-1$ occurrences of $U$, again with no occurrence of the pattern $UD$. As seen from the direct correspondence $M\to A$, $D\to L$, $U\to T$,
$c\to b$, $k_2\to k$ and $m_2\to m$ with the calculation presented in the proof of Proposition~\ref{prop:Fbkm},
the total number of words with the above requirement is now $q^{(c)}_{k_2}(m_2)$ as in~\eqref{eq:qkcdef}.
Note that $q^{(c)}_{k_2}(m_2)=0$ for $k_2\geq m_2+c+1$ as it should. 
Each word now codes for a lattice path of total length $2m_2-1$ and height difference $-(2c+1)$, see Figure~\ref{fig:pathstoforests2}. Calling $-\ell_2$ the minimal height of this path, with $\ell_2\geq 2c+1$ by construction, we can decompose the path into $2\ell_2-2c$ blocks which are Dyck paths coding for rooted plane trees.
As before, we transfer the markings of the down steps within a Dyck path to their associated non-root vertices in the associated plane tree. The absence of the $UD$ pattern then guarantees that \emph{all the (non-root) leaves of the trees are marked}. As for the root of the trees, we mark them if the block at hand is followed by a marked down step, which is possible only for the
$\ell_2$ first blocks in the decomposition.  Altogether, we arrive at the following:
\begin{prop}
  \label{prop:def_F2}
For fixed integers $c,k_2\geq 0$ and $m_2\geq c+1$, the number of sequences~$\mathcal S_2$ made of $2\ell_2-2c$ rooted plane trees with a total of $m_2+c-\ell_2$ internal edges, for some (unfixed) value $\ell_2$ with 
$2c+1\leq \ell_2 \leq m_2+c$ , with a total of $k_2$ marked vertices and such that all the (non-root) leaves of the trees are marked and the roots of the last $\ell_2-2c$ plane trees are not marked is given by 
$ q^{(c)}_{k_2}(m_2)$ as in \eqref{eq:qkcdef}.
\end{prop} 

We are now ready for the proof of Proposition~\ref{prop:pcmm}.
\begin{proof}[Proof of Proposition~\ref{prop:pcmm}]
As in \cite{polytightmaps}, we use the fact that a planar map with two faces can be built out of two sequences
of plane trees by sticking them together. Consider more precisely a sequence $\mathcal S_1$
of plane trees as defined in Proposition~\ref{prop:def_F1}
and call $2\ell_1+2c+2$ the length of this sequence, with $\ell_1\geq 0$.
Consider also a sequence $\mathcal S_2$ of plane trees as defined in Proposition~\ref{prop:def_F2}
and call $2\ell_2-2c$ the length of this sequence, with $\ell_2\geq 2c+1$.
As in Figure~\ref{fig:pathstoforests1} (resp.\ Figure~\ref{fig:pathstoforests2}),
we transform the sequence $\mathcal S_1$ (resp.\ $\mathcal S_2$) into a single connected
object by attaching the trees with a \emph{spine}
made of $2\ell_1+2c+1$ (resp. $2\ell_2-2c-1$) elementary edges.      

\begin{figure}
  \centering
  \includegraphics[width=.85\textwidth]{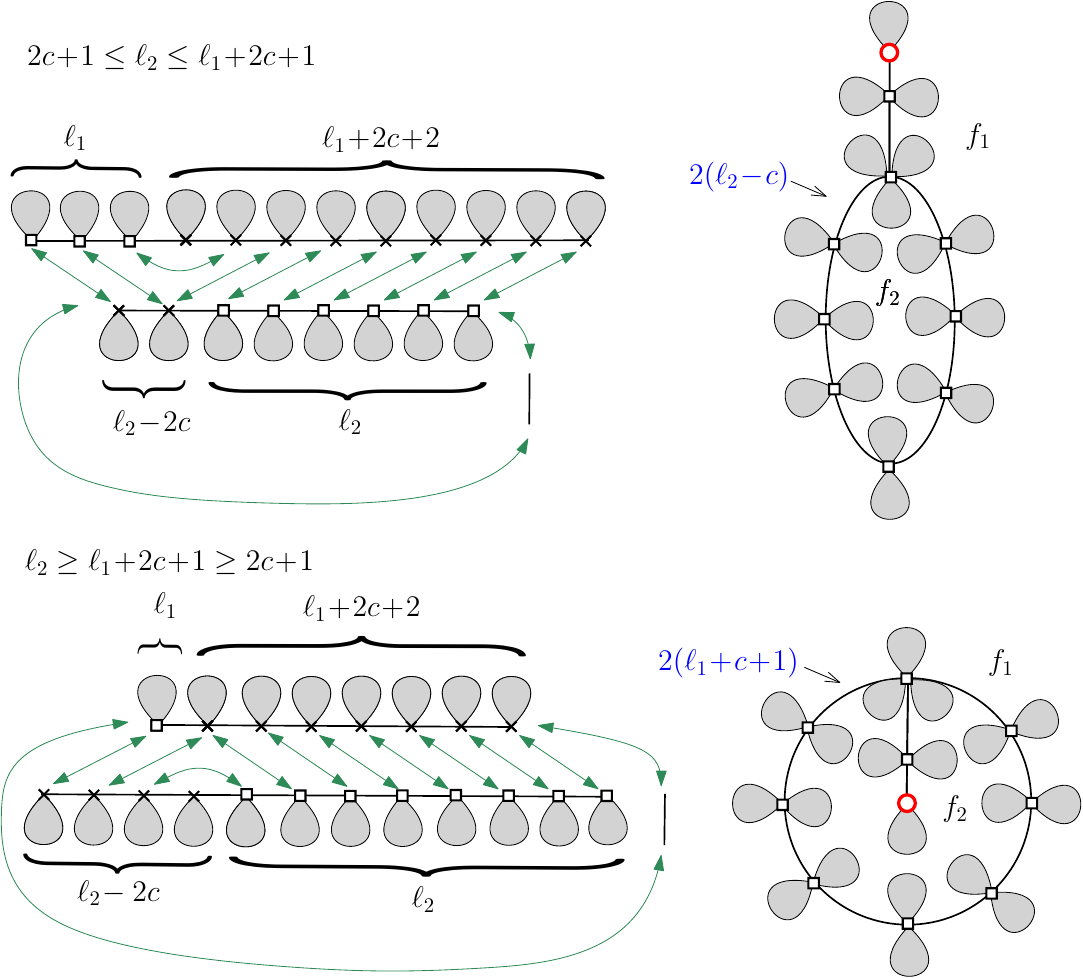}
  \caption{The gluing of two sequences of rooted trees into a two-face map (see text) according to whether
  $\ell_2$ lies between $2c+1$ and $\ell_1+2c+1$ (top) or is larger than or equal to $\ell_1+2c+1$ (bottom). The green arrows indicate which vertices to identify in the gluing process. In both case, the length of the separating loop (indicated in blue) is larger than or equal to $2(c+1)$.
  Squares indicate markable vertices and crosses non-markable ones.
  The red circle indicates the mandatory marking of the distinguished vertex.
  }
  \label{fig:twofaces}
\end{figure}
Assume first that $\ell_2\leq \ell_1+2c+1$ so that the spine of $\mathcal S_1$ is longer than (or of the same length as)
that of $\mathcal S_2$. We then shorten the spine of $\mathcal S_1$ by pulling up the $(\ell_1+1)$-th tree and by zipping the spine 
so as to attach the root of the $(\ell_1+1-j)$-th tree to that of the $(\ell_1+1+j)$-th one for $j=1,\ldots,\ell_1+2c+1-\ell_2$.
This creates a shorter spine of length $2\ell_2-2c-1$ with we may now glue ``face to face'' with the spine
of $\mathcal S_2$ and close it into a cycle of length $2(\ell_2-c)$ by adding an extra closing edge as shown in 
Figure~\ref{fig:twofaces}-top. The final result is a map with two faces $f_1$ and $f_2$ of 
respective degrees $2\times(m_1-(c+1)-\ell_1)+2\ell_1+2c+2=2m_1$ and $2\times(m_2+c-\ell_2)+2\ell_2-2c=2m_2$
with separating cycle of length $\mathcal{L}=2(\ell_2-c)$ with an additional distinguished marked vertex incident to $f_1$ corresponding to the endpoint of the zip, leading to a total of $k_1+k_2+1$ marked vertices (and all the leaves marked).
Note that the gluing procedure is such that a markable vertex is always glued to an unmarked one and the tip 
of the zip is initially not marked. These two conditions are crucial to ensure that all vertices are markable in the two-face map and that the construction is reversible without ambiguity as to which side to assign vertex markings after unzipping. Note also that the additional marking is needed to know where and how far to unzip.
Finally, since $\ell_2\geq 2c+1$, we deduce that $\mathcal{L}\geq 2(c+1)$.

Assume now that $\ell_2\geq \ell_1+2c+1$ so that the spine of $\mathcal S_1$ is shorter than (or of the same length as)
that of $\mathcal S_2$. We now shorten this latter spine by pulling down the $(\ell_2+1)$-th tree (counted from the right on Figure \ref{fig:twofaces}-bottom) and by zipping the spine so as to attach the root of the $(\ell_2+1-j)$-th tree to that of the $(\ell_2+1+j)$-th one for $j=1,\ldots,\ell_2-(\ell_1+2c+1)$. The two spines then have the same length $2\ell_1+2c+1$ and  we may again glue them and close the resulting segment into a cycle of length $2(\ell_1+c+1)$ by adding an extra closing edge. The final result is again a map with two faces $f_1$ and $f_2$ of 
respective degrees $2m_1$ and $2m_2$
with now a separating cycle of length $\mathcal{L}=2(\ell_1+c+1)$ with an additional distinguished marked vertex incident to the face $f_2$ (leading to a total of $k_1+k_2+1$ marked vertices).
Again the gluing procedure is such that a markable vertex is always glued to an unmarked one and the tip 
of the zip is initially not marked, as required to ensure that all vertices are markable in the two-face map and that the construction is reversible without ambiguity.
Finally, since $\ell_1\geq 0$, we deduce that, again, $\mathcal{L}\geq 2(c+1)$.

Note that if $\ell_2= \ell_1+2c+1$, both constructions are fully identical: this corresponds to the case where no zipping is necessary and the additional marked vertex is on the separating loop itself. 

This ends the proof of Proposition \ref{prop:pcmm}
by summing over $k_1$ and $k_2$ with $k_1+k_2=k$.
\end{proof}
\begin{figure}
  \centering
  \includegraphics[width=.78\textwidth]{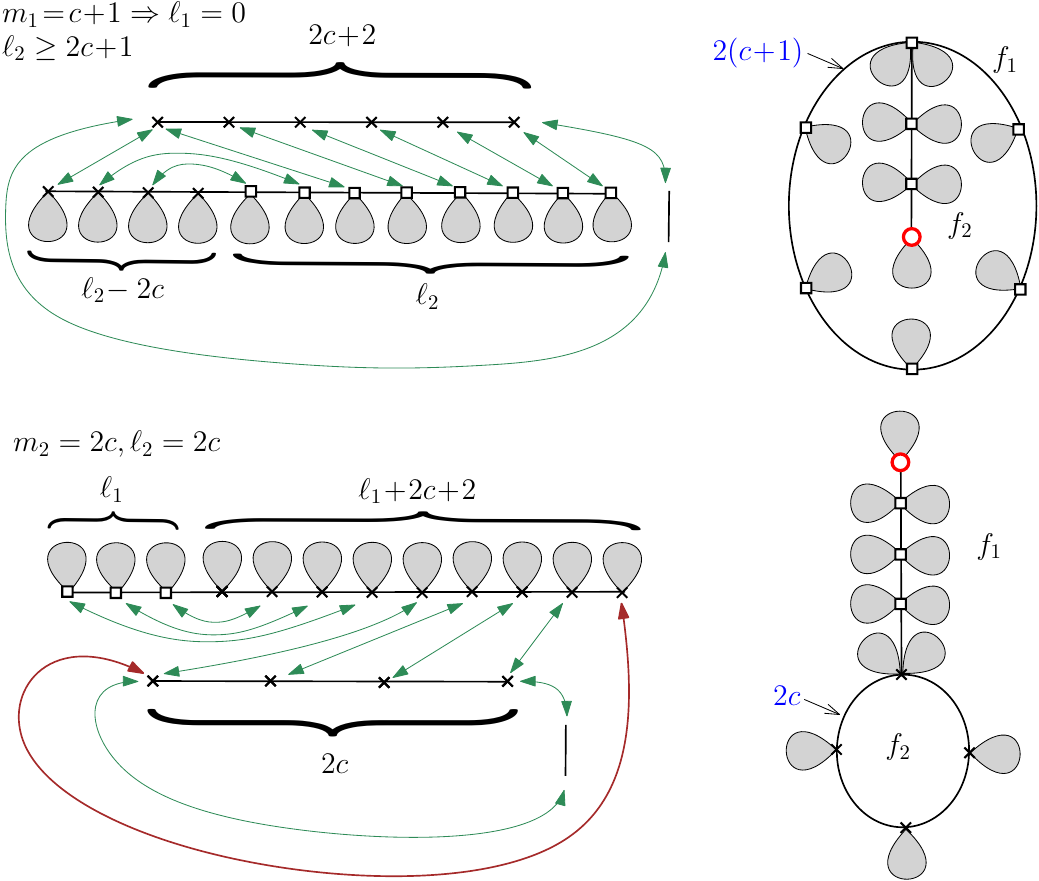}
  \caption{Top: limit case of the gluing of Figure~\ref{fig:twofaces}-bottom.
    Taking $m_1 = c+1$ implies that $\ell_1=0$ and yields a two-face map where the face $f_1$ is simple,
    and its incident vertices are all markable.
    This is counted by $q^{(c)}_{k}(m_2)$ if $f_2$ has degree $2m_2$ (and there is a total of $k+1$ markings, one being distinguished).
    Bottom: degenerate case $m_2=c$. The gluing now yields a two-face map where the face $f_2$ is simple,
    and its incident vertices \emph{are non-markable}.
    This degenerate case is not a particular case of
    the gluing of Figure~\ref{fig:twofaces}-top, since here the separating loop has length $2c$. The vertices of the empty spine are non-markable,
    and the zipping matches the rightmost tree of $\mathcal S_1$ with the leftmost vertex of $\mathcal S_2$ (see the brown arrow).
    This is counted by $p^{(c)}_{k}(m_1)$ if $f_1$ has degree $2m_1$ (and there is a total of $k+1$ markings, one being distinguished).}
      \label{fig:twofaces_limit}
\end{figure}

\paragraph{Remarks.}
We end this section with some remarks regarding the quantity $p^{(c)}_{k}(m_1,m_2)$.
\begin{rem}
  \label{rem:qkm_interpretation}
  In the case $m_1 = c+1$, we have $p^{(c)}_{k_1}(m_1) = \delta_{k_1,0}$,
  so that~\eqref{eq:pcmm} reduces, via $k_2 = k$,
  to 
  \begin{equation}
    \label{eq:pkccp1}
  p^{(c)}_{k}(c+1,m_2) = q^{(c)}_{k}(m_2) .
  \end{equation}
This gives a direct interpretation of $q^{(c)}_{k}(m_2)$ which can be recovered as follows:
  we already know that $q^{(c)}_{k}(m_2)$ counts the sequences $\mathcal S_2$ of Proposition~\ref{prop:def_F2}.
  The requirement $m_1=c+1$ forces that $\ell_1 = 0$ (recall that $0\leq \ell_1\leq m_1-(c+1)$), and that $\mathcal S_1$ is
  the \emph{empty spine} of length $2c+1$, composed of $2c+2$ non-markable vertices,
  with all the trees reduced to their root vertex.
  When zipping $\mathcal S_2$ to match it (see Figure~\ref{fig:twofaces_limit}, top),
  we end up with a map with a face of degree $2m_2$ and a \emph{simple} face of degree $2c+2$,
  where the cycle is also the contour of this simple face,
  and all the vertices of this cycle are markable.
  As before, this map has $k+1$ markings, one of them distinguished,
  and all its leaves are marked. 
\end{rem}

\medskip
\begin{rem}
  \label{rem:pkm_interpretation}
  Proposition~\ref{prop:pcmm} assumes that $m_1, m_2 > c$.
  In the case $m_2 = c$, we have $q^{(c)}_{k_2}(m_2) = \delta_{k_2,0}$,
  so that~\eqref{eq:pcmm} reads, via $k_1= k$, 
   \begin{equation}
   p^{(c)}_{k}(m_1,c) = p^{(c)}_{k}(m_1) .
   \end{equation}
  This leads to a \emph{new interpretation} of $p^{(c)}_{k}(m_1,m_2)$ when $m_2 = c$, 
  from the following direct interpretation of $p^{(c)}_{k}(m_1)$:
 we already know that $p^{(c)}_{k}(m_1)$ counts the sequences $\mathcal S_1$ of Proposition~\ref{prop:def_F1}.
  If we now replace $\mathcal S_2$ by the empty spine of length $2c-1$ (composed of $2c$ non-markable vertices
  with all the trees reduced to their root vertex),
  and zip $\mathcal S_1$ to match it (see Figure~\ref{fig:twofaces_limit}, bottom),
  we end up with a map  with a face of degree $2m_1$ and a \emph{simple} face of degree $2c$,
  where the cycle is also the contour of this simple face,
  and all the vertices of this cycle are now \emph{non-markable}.
  As always, this map has $k+1$ markings, one of them distinguished,
  and all its leaves are marked.
  This interpretation of $p^{(c)}_{k}(m_1)$, or equivalently of $ p^{(c)}_{k}(m_1,m_2)$ when $m_2=c$, will be useful in the next remarks,
  as well as for our main theorem~\ref{thm:main-result}.
\end{rem}

\medskip
\begin{rem}
\label{rem:sympkc}
  We have the relation
  \begin{equation}
    \label{eq:qktopkc}
    q_k^{(c)}(m) = \sum_{i=0}^{k} \binom{2c+1}i p_{k-i}^{(c)}(m)
  \end{equation}
  which can be understood as follows: according to Remark~\ref{rem:pkm_interpretation},  $p_{k-i}^{(c)}(m)$ counts two-face maps with a total of $k+1-i$ marked vertices, one distinguished, with a simple face $f_0$ of degree $2c$ having no incident marked vertex, the other face being of degree $2m$. In particular, the distinguished marked vertex is not incident to $f_0$ and the branch leading from $f_0$ to this vertex has no-zero length.
  We may unzip the first edge of this branch so as to obtain a larger simple face $f'_0$ of degree $2c+2$ and a shorter (by one edge) branch. None of the vertices incident to $f'_0$ are marked, except
  possibly the incident vertex at the beginning of the new branch leading to the distinguished vertex (note that this branch may be reduced to the distinguished vertex itself). If we now mark $i$ among the $2c+1$ other vertices incident to $f'_0$, which can be done in $\binom{2c+1}i$ ways, we end up with a two-face map with a simple face of degree $2c+2$, the other face being still of degree $2m$, with a total of
  $k+1$ marked vertices, one of them distinguished and exactly $i$ marked vertices along its unique cycle
  deprived of its vertex incident to the branch leading to the distinguished vertex. Summing over $i$, this enumerates 
  precisely maps counted by $q_k^{(c)}(m)$ according to Remark~\ref{rem:qkm_interpretation}.

  From \eqref{eq:qktopkc}, we get the alternative expression
  \begin{equation}
    p_k^{(c)}(m_1,m_2) = \sum_{\substack{k_1,k_2,i \geq 0 \\ k_1+k_2+i=k}} \binom{2c+1}{i} p_{k_1}^{(c)}(m_1) p_{k_2}^{(c)}(m_2).
  \end{equation}
  which displays that $p_k^{(c)}(m_1,m_2)$ is indeed symmetric in its two arguments as it should.
\end{rem}

\medskip
\begin{rem}
  \label{rem:fixedcyclelength}
  For bookkeeping purposes, let us mention another approach to
  computing $p^{(c)}_k(m_1,m_2)$. Let us denote by
  $o^{(d)}_k(m_1,m_2)$ the number of planar tight maps with exactly
  two (labeled) faces of respective degrees $2m_1$ and $2m_2$, with
  $(k+1)$ among their vertices marked, one of them being distinguished,
  and with their unique cycle of length \emph{exactly} $2d$. As
  $p^{(c)}_k(m_1,m_2)$ corresponds to two-face maps whose cycle has
  length at least $2(c+1)$, we have, for $m_1,m_2\geq c+1$,
  \begin{equation}
    \label{eq:pksumcycle}
    p^{(c)}_k(m_1,m_2) = \sum_{d \geq c+1} o^{(d)}_k(m_1,m_2).
  \end{equation}
  Then, in terms of the same univariate polynomials $p_k^{(c)}(m)$ and
  $q_k^{(c)}(m)$ as above, we have the expression
  \begin{multline}
    \label{eq:fixedcyclelength}
    o^{(d)}_k(m_1,m_2) = 
    \sum_{\substack{\kappa_1,\kappa_2 \geq 1 \\ \kappa_1+\kappa_2=k+1}} \frac{(2d)(\kappa_1+\kappa_2)}{\kappa_1 \kappa_2} p^{(d)}_{\kappa_1-1}(m_1) q^{(d-1)}_{\kappa_2-1}(m_2) \\
    + p^{(d)}_k(m_1) \delta_{m_2, d} + \delta_{m_1, d}q^{(d-1)}_k(m_2)
  \end{multline}
  which may be combinatorially interpreted as follows. Consider a map
  contributing to $o^{(d)}_k(m_1,m_2)$: 
  cutting along its unique cycle and filling the holes with new simple faces $\tilde f_1$ and $\tilde f_2$ of degree $2d$,
  we get, upon transferring the markings to face $f_2$:
  on one hand a two-face map $\mathcal Q$ with a simple face $\tilde f_1$ and the initial face $f_2$,
  and on the other hand a two-face map $\mathcal P$ with a simple face $\tilde f_2$ and the initial face $f_1$ without markings incident to the face $\tilde f_2$.
  Calling $ \kappa_1 $ and $\kappa_2 $ the total number of markings in $\mathcal Q$ and $\mathcal P$ respectively,
  those maps are the two types of maps considered respectively in Remarks~\ref{rem:qkm_interpretation} (with $c = d-1, k = \kappa_1-1$)
  and~\ref{rem:pkm_interpretation} (with $c = d, k=\kappa_2-1$),
  except we miss the distinguished marked vertex in one of those faces.
  After correcting this problem,
  since there are $2d$ ways to reassemble the maps $\mathcal P$ and $\mathcal Q$,
  this gives a number of possibilities equal to $(2d) p^{(d)}_{\kappa_1-1}(m_1) q^{(d-1)}_{\kappa_2-1}(m_2)$.
  Doing so, we obtain a two-face map with a pair of distinguished vertices (one incident to $f_2$, and one incident to $f_1$ but not $f_2$).
  Starting from a two-face map without distinguished vertices, there are $\kappa_1 \kappa_2$ ways to choose such pairs,
  and $\kappa_1+\kappa_2$ ways to choose one distinguished vertex in the whole map,
  so, to correct the above problem, we multiply the preceding expression by $\frac{\kappa_1+\kappa_2}{\kappa_1 \kappa_2}$.
  Summing over $\kappa_1 + \kappa_2 = k+1$ we get the first term in Equation~\eqref{eq:fixedcyclelength}.

  The extra terms on the second line
  correspond to the pathological situation where there is no marked
  vertex incident to one of the faces, forcing the corresponding
sequence of trees to contain only trees reduced to their root vertex. We have
  checked by computer algebra that the expressions~\eqref{eq:pcmm} and
  \eqref{eq:pksumcycle} match for the first values of $k$, and one might
  look for a general algebraic proof, besides the combinatorial proof
  that we sketched here.
\end{rem}

\subsection{The final result: enumeration of tight irreducible maps}
\label{sec:lutfin}

We are now ready to prove Theorem~\ref{thm:main-result} via the following:

\begin{prop}
  \label{prop:finalresult}
  For integers $n \geq 3$, $b,c\geq 1$, $m_1, m_2 \geq c+1$, and
  $m_3,\ldots,m_n\geq b$, the number of planar bipartite tight maps
  with $n$ labeled faces of respective degrees $2m_1,\ldots,2m_n$
  which are essentially $2b$-irreducible (as defined in
  Proposition~\ref{prop:tightdecomp}) and have separating girth at
  least $2(c+1)$ is equal to
  \begin{equation}
    \label{eq:lutfin}
    (n-3)! \sum_{k,k_3,\ldots,k_n \geq 0} p^{(c)}_k(m_1,m_2) q^{(b)}_{k_3}(m_3) \cdots q^{(b)}_{k_n}(m_n) \alpha^{(b)}_{k+k_3+\cdots+k_n,n-3}.
  \end{equation}
\end{prop}
\noindent Note again that $\alpha_{k,n}^{(b)}=0$ for $k>n$ hence the sum in \eqref{eq:lutfin} is \emph{finite}.

\begin{proof}
  By Proposition~\ref{prop:tightdecomp}, we need to enumerate tuples
  of the form
  $(\mathbf{m}_{12},\mathbf{s}_1,\ldots,\allowbreak \mathbf{s}_{k+1})$
  as in Proposition~\ref{prop:annulbij}, where furthermore the unique
  cycle of $\mathbf{m}_{12}$ has length at least $2(c+1)$, and where
  $\mathbf{s}_1,\ldots,\mathbf{s}_{k+1}$ are
  $2b$-irreducible. Proposition \ref{prop:pcmm} enumerates the former
  and Corollary~\ref{cor:Fbkm} the latter, upon changing $n$ into
  $n-2$ and shifting the face labels by $2$.  By summing over $k$, the
  wanted number reads
  \begin{equation}
    \sum_{k \geq 0} p^{(c)}_k(m_1,m_2) F^{(b)}_k(m_3,\ldots,m_n)
  \end{equation}
  which yields \eqref{eq:lutfin} by~\eqref{eq:Fbkm}.
\end{proof}

We may now proceed to the:

\begin{proof}[Proof of Theorem~\ref{thm:main-result}]
  The expression~\eqref{eq:main-result} for
  $\mathcal{N}_{n}^{(b)}(2m_1,\ldots,2m_n)$ is obtained:
  \begin{itemize}
  \item in the case $m_2 > b$ as a direct corollary of
    Proposition~\ref{prop:finalresult}, using the fifth item in
    Proposition~\ref{prop:tightdecomp} and the
    expression~\eqref{eq:pcmm} for $p^{(c)}_k(m_1,m_2)$ with $c=b$,
  \item in the case $m_2 = b$ as a direct consequence of the sixth
    item in Proposition~\ref{prop:tightdecomp} and of the
    interpretation of Remark~\ref{rem:pkm_interpretation} for
    $p^{(c)}_k(m_1,m_2)$ when $m_2=c=b$.
  \end{itemize}
  Let us now verify the polynomiality properties of
  $\mathcal{N}_{n}^{(b)}(2m_1,\ldots,2m_n)$.
  From their expressions \eqref{eq:qkcdef} and \eqref{eq:pkcdef} with $c \to b$, $p^{(b)}_{k_i}(m_i)$ and $q^{(b)}_{k_i}(m_i)$ are polynomials of total degree $2k_i$ in $b$ and $m_i$ (and even in $m_i$). Recall that $\alpha_{k,n-3}^{(b)}$ vanishes for $k>n-3$ and that for $k \leq n-3$ it is by Corollary~\ref{cor:alpha-polynomial} or Proposition~\ref{prop:alphapolexp} a polynomial of degree $2(n-3-k)$ in $b$.
  Therefore, each non-zero term in the sum in \eqref{eq:main-result} is of degree $2(n-3 - \sum_{i} k_i) + 2\sum_{i} k_i = 2n - 6$ in $b$ and the $m_i$'s. Its top degree term in the $m_i$'s is $\prod_i \frac{m_i^{2k_i}}{(k_i!)^2}$ which, for $\sum_{i} k_i=n-3$, is not present in any other non-zero term. Thus, the total degree of $\mathcal{N}_{n}^{(b)}(2m_1,\ldots,2m_n)$ is exactly $2n-6$.
 The expression \eqref{eq:main-result} is clearly symmetric in $m_2,\ldots,m_n$, and also 
 symmetric upon exchanging $m_1$ and $m_2$, as proved in Remark~\ref{rem:sympkc}, 
hence it is symmetric in all its arguments.
\end{proof}

The formula~\eqref{eq:main-result} given in this paper can be identified as \cite[Equation~(66)]{Budd2022}, as discussed in Appendix~\ref{app:compatibility-Budd}.

\paragraph{Particular cases.} Using $q^{(b)}_{k}(b)=p^{(b)}_{k}(b+1)=\delta_{k,0}$, we have the following specializations:
 \begin{prop}
  For $n \geq 3$ and $m_1,m_2 \geq b$ not both equal to $b$, we have
  \begin{gather}
    \mathcal{N}_{n}^{(b)}(2m_1,2m_2,\underbrace{2b,\ldots,2b}_{n-2 \text{ times}}) =
    (n-3)! \sum_{k=0}^{n-3} p_k^{(b)}(m_1,m_2) \alpha_{k,n-3}^{(b)}.
  \end{gather}
\end{prop}
  
 \begin{prop}
  For $n \geq 3$ and $m>b$, we have
  \begin{gather}
    \mathcal{N}_{n}^{(b)}(2m,\underbrace{2b,\ldots,2b}_{n-1 \text{ times}}) =
    (n-3)! \sum_{k=0}^{n-3} p_k^{(b)}(m) \alpha_{k,n-3}^{(b)}, \\
    \mathcal{N}_{n}^{(b)}(2m,2b+2,\underbrace{2b,\ldots,2b}_{n-2 \text{ times}}) =
    (n-3)! \sum_{k=0}^{n-3} q_k^{(b)}(m) \alpha_{k,n-3}^{(b)} .
  \end{gather}
\end{prop}

 \begin{prop}
  For $n \geq 3$, we have
  \begin{gather}
    \mathcal{N}_{n}^{(b)}(2b+2,\underbrace{2b,\ldots,2b}_{n-1 \text{ times}}) =
    (n-3)! \alpha_{0,n-3}^{(b)}.
   \end{gather}
\end{prop}
For $b=2$ and $n \geq 3$, this yields
\begin{gather}
    \mathcal{N}_{n}^{(2)}(6,\underbrace{4,\ldots,4}_{n-1 \text{ times}}) =
    \frac{(2(n-3))!}{(n-3)!}
   \end{gather}
for the number of $4$-irreducible dissections of the hexagon by $(n-1)$ labeled quadrangles, in agreement\footnote{
  up to a $\frac 6 {(n-1)!}$ factor, since these papers consider maps with a rooted boundary (factor of $6$)
  and undistinguished quadrangles (factor of $\frac 1 {(n-1)!}$).
} with
\cite{MS68} and \cite{FPS08}. Indeed, the tightness (and even the $2$-connectedness) of the map follows automatically from the irreducibility constraints.

Finally, recalling Remark~\ref{rem:singleslice} (with $n \to n+2$) and
Corollary~\ref{cor:Fbkm}, we get:
\begin{prop}
  \label{prop:slice}
  For $n \geq 1$ and $m_1,\ldots,m_n \geq b$, we have
  \begin{multline}
    \mathcal{N}_{n+2}^{(b)}(2b+2,2b,2m_1,\ldots,2m_n)= F_0^{(b)}(m_1,\ldots,m_n) \\
    = (n-1)! \sum_{k_1,\ldots,k_n \geq 0} q^{(b)}_{k_1}(m_1) \cdots q^{(b)}_{k_n}(m_n) \alpha^{(b)}_{k_1+\cdots+k_n,n-1}.
  \end{multline}
\end{prop}

\section{Conclusion}
\label{sec:conc}

In this paper, we gave a fully combinatorial proof of Theorem~\ref{thm:main-result}
counting the number of planar bipartite tight $2b$-irreducible maps with labelled faces of prescribed degrees.
As opposed to Budd's approach in \cite{Budd2022} based on a substitution of formal power series,
we had recourse here to a decomposition of desired maps into (tight $2b$-irreducible) slices
for which we presented a bijective enumeration based on $b$-decorated trees drawn on their derived map.

Even though we have not worked out the details, we believe that the assumption of bipartiteness is not essential,
and our formula can be extended to the enumeration of planar tight $d$-irreducible maps with labelled faces of arbitrary odd or even degrees.
The case $d=0$ was treated in detail in \cite[Theorem 2.12]{polytightmaps} and involves quasi-polynomials instead of ordinary polynomials.

The decorated tree formulation of Section~\ref{sec:treeformulation}
involves placing arrows on the primal half-edges of the map, which
form arrow trees.  Such construction is strongly reminiscent of the
approach of Bernardi and Fusy, which consists in choosing an
appropriate biorientation of the map.  We noted in
Remark~\ref{rem:altrepr} that, in the absence of special vertices
(which corresponds to maps with no face of degree $2b$), our
$b$-decorated trees correspond precisely to the $(b+1)$-dibranching
mobiles of \cite[Definition 8]{BF12b}.  Reintroducing the faces of
degree $2b$ in their formulation seems to add an extra challenge, and
we wonder whether the ideas of \cite{BFL23_SchnyderWoods}, which
involve working with orientations of the \emph{derived} map, could
solve this issue.

Finally, we cannot help but notice that the general structure of the right-hand side of Equation~\eqref{eq:main-result}
is very reminiscent of \cite[Equation (32)]{BGM24_Zhukovsky} for planar bipartite maps with a number of labelled \emph{tight boundaries} of prescribed degrees (and arbitrary even-valent internal faces).

\appendix

\section{Going back from a decorated tree to a $\boldsymbol 0$-slice}
\label{app:closing-tree}
We discuss here how to get back from a $b$-decorated tree for some $b\geq 1$ to the associated 
tight $2b$-irreducible $0$-slice. In practice we rather construct the dual map of the $0$-slice
by a classical procedure of \emph{closure} consisting in matching the leaflets of the decorated tree to properly
placed \emph{buds} via a system of non-crossing arches.

The first step of the construction consists in adding buds to the decorated tree as follows (see Figure~\ref{fig:refermeture}).
We first complete the decorated tree by adding a dual half-edge from its univalent root edge-vertex $v_0$
to a dual vertex $\Delta_0$ in the outer face.
Then, all edge-vertices have degree two 
in the (completed) decorated tree while they had degree four in the derived map. The added buds then correspond to the
missing incident dual half-edges. More precisely, an edge-vertex $v$ of the decorated tree may be a 
bioriented edge-vertex and we then attach to $v$ two buds, one in each corner around $v$ in the decorated tree (to account for
the fact that $v$ should be incident to two dual half-edges not covered by the tree in the derived map).
The vertex $v$ may be a bent edge-vertex: we
then attach one bud to $v$ (to account for
the fact that $v$ should now be incident to one dual half-edges not covered by the tree in the derived map).
The sector in which we put the bud is determined by demanding that the order of appearance clockwise around $v$ be dual half-edge/bud/primal half-edge,
where the primal or dual nature of each half-edge is determined from the nature of its incident vertex other than the edge-vertex.
This applies in particular to the twig-vertices.
Finally, it may be that $v$ is (special) dual/dual edge when $b=1$, in which case we add no bud. 
Altogether, this yields $2$ buds for each bioriented edge-vertex, $1$ bud for each bent edge-vertex, and $0$ bud for each dual/dual edge-vertex.

\begin{figure}[t]
  \centering
  \includegraphics[width=\textwidth]{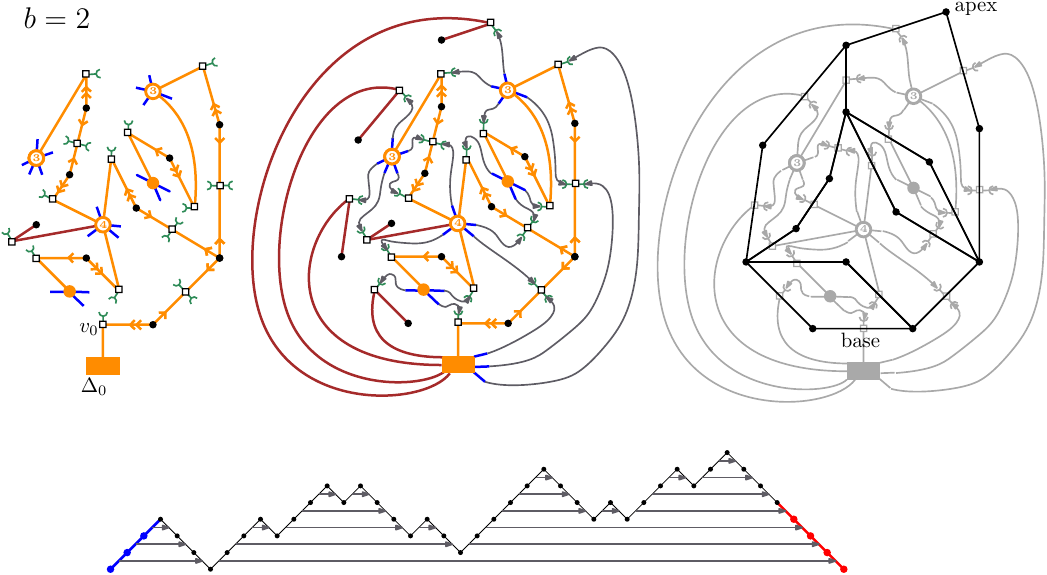}
  \caption{Constructing a $2b$-irreducible $0$-slice from a $b$-decorated tree (here $b=2$). Top left: we first
  equip the tree with buds attached to bioriented edge-vertices (two buds each, green) and to bent edge-vertices (one bud each).
  Bottom: the path coding for the sequence of leaflets and buds
  read counterclockwise around the tree from $\Delta_0$ (black steps).
  The path is (minimally) completed into a Dyck path by adding an initial series of $l=3$ up steps (in blue)
  and a final series of $l+1=4$ down steps. We also indicated (dark grey arrows) the canonical matching of up and down steps. 
  Top center: we add accordingly a sequence of $l$ leaflets and one of $l+1$ twigs (with buds) attached to $\Delta_0$.
  Each leaflet is then canonically matched to a subsequent bud counterclockwise around the
  tree by a system of non-crossing arches (in dark grey). Top right: removing the primal edges yields the
  dual map (in grey) of the desired $0$-slice (in black).}
  \label{fig:refermeture}
\end{figure}

Consider now a \emph{descending subtree} $\mathcal{T}$ of the decorated tree, which codes for a $2p$-(sub)slice for some $p$ between $0$ and $b$,
and which we root at the parent half-edge of $\Delta(BC)$ if the $2p$-slice has base $BC$ (this parent edge is a primal 
half-edge with $p$ descending arrows if $1 \leq p \leq b$ or a dual edge without arrows if $p=0$).
We define the \emph{charge} $c(\mathcal{T})$ of the subtree $\mathcal{T}$ as the total number of incident 
leaflets minus that of incident buds.
\begin{prop}
We have
\begin{equation}
c(\mathcal{T})=\left\{
\begin{matrix}
2p\ \ & \hbox{if\  $1\leq p\leq b$}, \\
 1\ \  &\hbox{if\  $p=0$}.
\end{matrix}
\right. 
\end{equation}
\end{prop}
\begin{proof}
Assume $b>1$. The above property is proved by induction on the three cases that arise:
\begin{itemize}
  \item[]{(I)} $\mathcal{T}$ is a subtree coding for a $2p$-slice with $p<b$, excluding case (III);
  \item[]{(II)} $\mathcal{T}$ is a subtree coding for a $2b$-slice;
  \item[]{(III)} $\mathcal{T}$ is a subtree coding for a $2(b-1)$-slice that is a $2b$-angle.
\end{itemize}
 
In case (I), $\Delta(C)$ is an internal vertex of the tree with $q$ descending subtrees coding for $2p_i$-slices
with $p_i\geq 1$ ($1\leq i\leq q$) and  $p_1+\cdots +p_q=(p+1)$. From the induction hypothesis, 
those contribute a total of $2p_1+\cdots+2p_q$
to $c(\mathcal{T})$. If $0<p<b$, $\Delta(BC)$ is a bioriented edge-vertex and gives rise to two buds hence contributes $-2$ to $c(\mathcal{T})$, so that
 $c(\mathcal{T})=2(p_1+\cdots +p_q)-2=2p$ as wanted. If $p=0$, $\Delta(BC)$ is a
 bent edge-vertex and gives rise to one bud hence contributes $-1$ to $c(\mathcal{T})$, so that
 $c(\mathcal{T})=2(p_1+\cdots +p_q)-1=2(0+1)-1=1$ as wanted (in fact we necessarily have $q=1$ and $p_1=1$ in this case).  
 
 In case (II), $\Delta(BC)$ is a bent edge-vertex and contributes $-1$ to $c(\mathcal{T})$. 
 The descending subtree starts with a labeled vertex of degree $2m$ with $m-b-1$ twigs, each contributing
 $-1$ to $c(\mathcal{T})$ and a total of $m+b$ leaflets or subtrees coding for $0$-slices, each contributing $+1$
 (from the induction hypothesis).
 This leads to $c(\mathcal{T})=-1-(m-b-1)+(m+b)=2b$ as wanted.
 
 In case (III), $\Delta(BC)$ is a bent edge-vertex and contributes $-1$ to $c(\mathcal{T})$. 
 The descending subtree is a single special vertex carrying $2b-1$ leaflets. This leads to 
  $c(\mathcal{T})=-1+2b-1=2(b-1)$ as wanted.  
  This ends the proof for $b>1$. 
  
For $b=1$, the only difference is in case (III) where $\Delta(BC)$ is a dual/dual 
edge-vertex and contributes $0$ to $c(\mathcal{T})$. The descending subtree is reduced to a special dual vertex
with $1$ leaflet, hence $c(\mathcal{T})=0+1=1$ as wanted since the $2$-angle map is a $0$-slice. 
\end{proof}

The main lesson of the above proposition is that, if we now consider the whole decorated tree 
coding for a $0$-slice, completed with the dual half-edge connecting $v_0$ to $\Delta_0$,
this tree, when equipped with buds, has charge $+1$ hence has exactly one more incident leaflet
than buds.  More precisely, reading the sequence of leaflets and buds encountered by going 
counterclockwise around the tree from $\Delta_0$ gives a two-letter word in the alphabet 
$\{+1,-1\}$ (with the correspondence leaflet $\to +1$ and bud $\to -1$) with one more $+1$ than $-1$. This in turn can be represented 
as a two-step path with up ($+1$) and down ($-1$) steps, starting from height $0$ and ending at height $+1$.
Calling $-l$ the minimum height, with $l\geq 0$, we can transform this path into a Dyck path
from height $0$ to height $0$, and staying above or at height $0$, by completing the path with $l$ 
preceding up steps and $l+1$ following down steps and by shifting the heights by $+l$. Now,
as it is well-known for Dyck paths, each ascending step can be matched
canonically to a subsequent descending step at the same height on its right (see Figure~\ref{fig:refermeture}-bottom). 

Transposed to the bud-equipped completed decorated tree, this means that, if we further equip it by $l$ leaflets and 
$l+1$ twigs attached to $\Delta_0$
(with a bud attached to each new added twig) and read the sequence
of leaflets and buds counterclockwise around the completed tree, there is a canonical matching
where each leaflet is connected to a subsequent bud by an arch so that the arch system is non-crossing,
does not cross the tree, and no arch passes below $\Delta_0$ (see Figure~\ref{fig:refermeture}, top). 
These arches reconstruct precisely the missing dual half-edges so that, if we now
remove from the tree the primal half-edges (i.e.\ the half-edge carrying arrows and the primal half-edges in the twigs), 
and erase all the edge-vertices,
we find the dual map of the $0$-slice that we are looking for. The apex of the $0$-slice is dual to the external face
while its base is dual to edge of the dual map that contained $v_0$ (see Figure~\ref{fig:refermeture}, top right).

\section{Compatibility of formula~(\ref{eq:main-result}) with Budd's expression}
\label{app:compatibility-Budd}
We start from \eqref{eq:main-result}, and express $\alpha_{k,n}^{(b)}$ via~\eqref{eq:arrowtreedef}.
This leads to:
\begin{equation}
  \begin{split}
  \mathcal{N}_{n}^{(b)}(2m_1,\ldots,2m_n)
  &= (n-2)! [z^{n-2}] \mathcal S \\
  \hbox{with } \mathcal S &= \sum_{k_1,\ldots,k_n \geq 0} p^{(b)}_{k_1}(m_1) q^{(b)}_{k_2}(m_2) \cdots q^{(b)}_{k_n}(m_n)
   \frac{U_0(z)^{1+\sum\limits_{i\geq 1} k_i}}{1+\sum\limits_{i\geq 1} k_i}
  \end{split}
\end{equation}
We now invert the relation \eqref{eq:qktopkc} with $c \to b$, namely:
\begin{equation}
  \label{eq:pktoqkc}
  p_k^{(b)}(m) = \sum_{j=0}^{k} (-1)^j \binom{2b+j}j q_{k-j}^{(b)}(m)
\end{equation}
which leads to:
\begin{equation}
  \begin{split}
  \mathcal S &= \sum_{k_1,\ldots,k_n \geq 0} \sum_{j=0}^{k_1} (-1)^j \binom{2b+j}j q^{(b)}_{k_1-j}(m_1) q^{(b)}_{k_2}(m_2) \cdots q^{(b)}_{k_n}(m_n)
    \frac{U_0(z)^{1+\sum\limits_{i\geq 1} k_i}}{1+\sum\limits_{i\geq 1} k_i} \\
    &=   \sum_{k_0, k_1,\ldots,k_n \geq 0} (-1)^{k_0} \binom{2b+k_0}{k_0} q^{(b)}_{k_1}(m_1) \cdots q^{(b)}_{k_n}(m_n)
    \frac{U_0(z)^{k_0 + 1+\sum\limits_{i\geq 1} k_i}}{k_0+1+\sum\limits_{i\geq 1} k_i} 
  \end{split}
\end{equation}
through renaming $(j, k_1) \to (k_0, k_0+k_1)$. We now deduce: 
\begin{equation}
  \begin{split}
    \mathcal S &= \int_0^{U_0(z)} dr \sum_{k_0, k_1,\ldots,k_n \geq 0} (-1)^{k_0} \binom{2b+k_0}{k_0} r^{k_0+\sum\limits_{i\geq 1} k_i} q^{(b)}_{k_1}(m_1) \cdots q^{(b)}_{k_n}(m_n) \\
    &= \int_0^{U_0(z)} \frac{dr}{(1+r)^{2b+1}} \sum_{k_1,\ldots,k_n \geq 0} r^{\sum\limits_{i\geq 1} k_i} q^{(b)}_{k_1}(m_1) \cdots q^{(b)}_{k_n}(m_n) \\
    &= \int_0^{U_0(z)} \frac{dr}{(1+r)^{2b+1}} \prod_{i=1}^{n}I(b, m_i; r)
  \end{split}
\end{equation}
with $I(b, m; r) := \sum_{k=0}^{+\infty} r^k q^{(b)}_{k} (m)$ as in
\cite[Equation~(48)]{Budd2022}.  This reproduces precisely the result
of [\emph{op. cit.}, Equation~(66)] upon identifying this reference's
$J^{-1}(b; z)$ to our $U_0(z)$: indeed, recall from
\eqref{eq:U0implicit} that $U_0(z)$ is determined implicitly by
$z=h_b(U_0(z))$, and note that our $h_b(u)$ is the same as $J(b;u)$ as
defined in~[\emph{op. cit.}, Equation~(49)].

\section{Counting planar $2b$-irreducible $2b$-angulations with $n$ faces}
\label{sec:Npatho}

In this appendix, we consider the problem of counting planar
$2b$-irreducible $2b$-angula\-tions, that is $2b$-irreducible maps whose
all faces have degree $2b$. Note that such maps are automatically
bipartite and tight. We noted after Theorem~\ref{thm:main-result} that
it does not quite hold in the case where all $m_i$ are equal to $b$,
for there is an pathological term as visible in~\cite[Theorem
1]{Budd2022}. Precisely, we have the following:
\begin{prop}
  \label{prop:Npatho}
  For any $n \geq 3$, the number of $2b$-irreducible $2b$-angulations
  with $n$ labeled faces is equal to
  \begin{equation}
    \label{eq:Npatho}
    \mathcal{N}_{0,n}^{(b)}(\underbrace{2b,\ldots,2b}_{n \text{ times}}) =
    (n-3)! \sum_{k=0}^{n-3} (-1)^k \binom{2b+k}k \alpha_{k,n-3}^{(b)} + \mathbf{1}_{n \geq 4} \frac{(n-1)!}{2} (-1)^n.
  \end{equation}
\end{prop}
Let us observe that the first term in the right-hand side is equal to
the right-hand side of~\eqref{eq:main-result} for $m_1=\cdots=m_n=b$,
since $p_k^{(b)}(b)=(-1)^k \binom{2b+k}k$ and
$q_k^{(b)}(b)=\delta_{k,0}$. Here it is slightly puzzling that the
sign alternates with $k$, whereas under the assumptions of
Theorem~\ref{thm:main-result} all the terms in~\eqref{eq:main-result}
were non-negative. In the remainder of this appendix, we give a
combinatorial, but unfortunately not bijective, derivation of
Proposition~\ref{prop:Npatho}. It is based on an argument which was
mentioned, but not detailed, at the end
of~\cite[Section~9.3]{irredmaps}.

The bijection described in Section~\ref{sec:slicedecgen} does not
easily allow to count $2b$-irreducible $2b$-angulations: in the two
items characterizing $2b$-irreducible maps at the end of
Proposition~\ref{prop:tightdecomp}, it is always assumed that face $1$
has a degree strictly larger than $b$. To circumvent this issue, we
first count closely related objects by specializing
Proposition~\ref{prop:finalresult} to the case $c=b-1$,
$m_1=\ldots=m_n=b$: noting that, by~\eqref{eq:pkccp1}, we have
$p_k^{(b-1)}(b,b)=q_k^{(b-1)}(b)=\binom{2b-1}k$, we obtain the
following:
\begin{lem}
  For any $n \geq 3$, the number of
  \underline{essentially} $2b$-irreducible $2b$-angulations with $n$
  labeled faces with separating girth equal to $2b$ is equal to
  \begin{equation}
    \beta^{(b)}_n := (n-3)! \sum_{k \geq 0} \binom{2b-1}k \alpha^{(b)}_{k,n-3}.
  \end{equation}
\end{lem}

Note that, by~\eqref{eq:arrowtreedef}, we have
\begin{equation}
  \sum_{n \geq 3} \frac{\beta^{(b)}_n}{(n-3)!} z^{n-3} = \sum_{k \geq 0} \binom{2b-1}k U_0'(z) U_0(z)^k = U_0'(z) (1+U_0(z))^{2b-1}
\end{equation}
with $U_0(z)$ as in Section~\ref{sec:laginv} (recall that it depends
implicitly on $b$). Integrating over $z$ and multiplying by $2b$, we
deduce that $(1+U_0(z))^{2b}$ is the generating function of
essentially $2b$-irreducible $2b$-angulations with two marked faces
$1$ and $2$, such that face $1$ is rooted and such that the separating
girth is equal to $2b$, counted with a weight $z$ per unmarked
face. In what follows, we denote the set of such maps by
$\mathfrak{M}_b$, and by $\mathfrak{N}_b$ its subset consisting of
maps which are $2b$-irreducible in the strong (not just essentially)
sense, which means that the only separating cycles of length $2b$ are
the contours of the marked faces. The following lemma corresponds
to~\cite[Equation~(9.21)]{irredmaps} specialized to the case of
$2b$-angulations:
\begin{lem}
  \label{lem:annularid}
  Let us denote by $N_{b}(z)$ the generating function of maps in
  $\mathfrak{N}_b$, counted with a weight $z$ per unmarked face.
  Then, we have
  \begin{equation}
    \label{eq:annularid}
    (1+U_0(z))^{2b} = \frac{1}{1-\left( N_{b}(z) - 1 - 2 b z \right) - b \frac{2z + z^2}{1+2z+z^2}}.
  \end{equation}
\end{lem}

\begin{proof}
  Let $\mathbf{m}$ be a map in $\mathfrak{M}_b$, i.e. a map
  contributing to $(1+U_0(z))^{2b}$, which we draw in the complex
  plane with the face $1$ containing the origin and the face $2$
  chosen as the outer face.  Let $C$ be the set of separating cycles
  of $\mathbf{m}$ with length $2b$. We say that a cycle $c \in C$ is
  \emph{cuttable} if it does not intersect any other cycle $c' \in C$,
  that is $c'$ remains within one of the two closed regions delimited
  by $c$. Upon cutting $\mathbf{m}$ along all cuttable cycles, we
  decompose it into a sequence of maps
  $(\mathbf{m}_1,\ldots,\mathbf{m}_\ell)$ nested into one another.
  These maps belong also to $\mathfrak{M}_b$, since they have
  naturally two marked faces and may be rooted in some canonical
  manner. All of them have at least one unmarked face.
  
  \begin{figure}
    \centering
    \includegraphics[width=\textwidth]{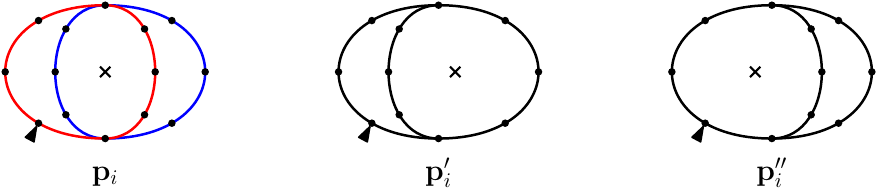}
    \caption{The pathological maps $\mathbf{p}_i$, $\mathbf{p}'_i$ and
      $\mathbf{p}''_i$, $i=1,\ldots,b$, in $\mathfrak{M}_b$ (here
      $b=4$ and $i=1$). The arrow indicates the root and the cross the
      marked inner face. In $\mathbf{p}_i$, there are two minimal
      separating cycles (shown in red and blue) which intersect each
      other, hence are not cuttable. The maps $\mathbf{p}'_i$ and
      $\mathbf{p}''_i$ contain no uncuttable cycles, but they are
      still pathological in the sense that, when cutting a map in
      $\mathfrak{M}_b$ along its cuttable cycles (as explained in the
      text) then we cannot find two consecutive elements
      $\mathbf{p}_i,\mathbf{p}'_i,\mathbf{p}''_i$ with the same index
      $i$.}
    \label{fig:PathologicalMapsBis}
  \end{figure}
  
  Now, we observe that for each $\mathbf{m}_j$, two situations may
  arise: either it contains no minimal separating cycle other than the
  contours of the marked faces, in which case it belongs to
  $\mathfrak{N}_b$, or it contains such a cycle, which necessarily
  intersects another minimal separating cycle as it would have been
  cuttable otherwise. In that case, using the essential
  $2b$-irreducibility, we see that $\mathbf{m}_j$ necessarily
  coincides with the ``pathological'' map $\mathbf{p}_i$ displayed on
  Figure~\ref{fig:PathologicalMapsBis} for some $i=1,\ldots,b$ (this
  index indicates the position of the root corner).

  We note that there is an extra constraint on the sequence
  $(\mathbf{m}_1,\ldots,\mathbf{m}_\ell)$: let $\mathbf{p}'_i$ and
  $\mathbf{p}''_i$ be the other ``pathological'' maps displayed on
  Figure~\ref{fig:PathologicalMapsBis}. While they belong to
  $\mathfrak{N}_b$, it is not possible to have within the sequence
  $(\mathbf{m}_1,\ldots,\mathbf{m}_\ell)$ two consecutive elements
  equal to $\mathbf{p}_i$, $\mathbf{p}'_i$ or $\mathbf{p}''_i$ with
  the same index $i$ (otherwise we would have either cut $\mathbf{m}$
  along an uncuttable cycle, or have in $\mathbf{m}$ a cycle of length
  $2b$ which is neither separating nor the contour of a face).

  We may code for this constraint in the following way: to the
  sequence $(\mathbf{m}_1,\ldots,\mathbf{m}_\ell)$ we associate a word
  $w$ of length $\ell$ over the alphabet $\{O,P_1,\ldots,P_b\}$, by
  replacing each $\mathbf{m}_j$ equal to $\mathbf{p}_i$,
  $\mathbf{p}'_i$ or $\mathbf{p}''_i$ for some $i$ by the letter
  $P_i$, and all other $\mathbf{m}_j$'s by the letter $O$. The word
  $w$ is such that two letters $P_i$ with the same index cannot appear
  consecutively. Splitting $w$ at each occurrence of $O$, it may be
  decomposed as $w_0Ow_1O\cdots O w_k$ for some $k \geq 0$ where
  $w_0,\ldots,w_k$ are so-called \emph{Smirnov words}
  \cite[Example~III.24]{Flajolet2009}, and have the multivariate
  generating function
  \begin{equation}
    S(v_1,\ldots,v_b) = \left( 1 - \sum_{i=1}^b \frac{v_i}{1+v_i} \right)^{-1}
  \end{equation}
  where $v_i$ is the weight per occurrence of the letter
  $P_i$. Attaching a weight $u$ per occurrence of $O$, we deduce that
  the generating function of the words $w$ of the wanted form is equal
  to
  \begin{equation}
    \sum_{k \geq 0} S(v_1,\ldots,v_b)^{k+1} u^k = \frac{S(v_1,\ldots,v_b)}{1-uS(v_1,\ldots,v_b)}
    = \frac{1}{1-u-\sum_{i=1}^b \frac{v_i}{1+v_i}}.
  \end{equation}
  The generating function of the sequences
  $(\mathbf{m}_1,\ldots,\mathbf{m}_\ell)$ obeying our constraint is
  then obtained by substituting $u=N_b(z)-1-2bz$ (corresponding to the
  generating function of maps in $\mathfrak{N}_b$ having at least one
  unmarked face and different from
  $\mathbf{p}'_1,\mathbf{p}''_1,\ldots,\mathbf{p}'_b,\mathbf{p}''_b$)
  and, for all $i=1,\ldots,b$, $v_i=2z+z^2$ (corresponding to the
  combined weights of $\mathbf{p}_i,\mathbf{p}'_i,\mathbf{p}''_i$).
  This yields the right-hand side of~\eqref{eq:annularid}, and we may
  conclude by verifying that the mapping
  $\mathbf{m} \mapsto (\mathbf{m}_1,\ldots,\mathbf{m}_\ell)$ is a
  bijection.
\end{proof}

\begin{proof}[End of the proof of Proposition~\ref{prop:Npatho}]
  By Lemma~\ref{lem:annularid}, we have
  \begin{equation}
    \begin{split}
      N_b(z) &= 2 - (1+U_0(z))^{-2b} + 2 b z - b \frac{2z+z^2}{1+2z+z^2} \\
      &= 1 - \sum_{i \geq 1} (-1)^i \binom{2b-1+i}i U_0(z)^i + b \frac{z^2(3+2z)}{(1+z)^2}.
    \end{split}
  \end{equation}
  Extracting the coefficient of $z^{n-2}$ for $n \geq 3$ and
  using~\eqref{eq:arrowtreedef}, we get
  \begin{equation}
    \begin{split}
      [z^{n-2}] N_b(z) &= \sum_{i \geq 1} (-1)^{i-1} \binom{2b+i-1}i \frac{i}{n-2} \alpha_{i-1,n-3} + b (-1)^n (n-1) \mathbf{1}_{n \geq 4} \\
      &= \frac{2b}{n-2} \sum_{k \geq 0} (-1)^k \binom{2b+k}k \alpha_{k,n-3} + b (-1)^n (n-1) \mathbf{1}_{n \geq 4}.
    \end{split}
  \end{equation}
  We obtain the wanted expression~\eqref{eq:Npatho} by multiplying by
  $(n-2)!$ to label the $n-2$ unmarked faces, and dividing by $2b$ to
  remove the rooting of face $1$.
\end{proof}

\printbibliography

\end{document}